\documentclass[11pt]{article}
\usepackage [dvips]{graphics}
\usepackage[centertags]{amsmath}
\usepackage{amsfonts}
\usepackage{amssymb}
\usepackage{amsthm}
\usepackage{newlfont}
\usepackage{stmaryrd}
\usepackage{color,epsfig}
\usepackage{amsfonts,graphicx,psfrag}
\usepackage{graphicx}
\usepackage{float}
\usepackage{subcaption}

\usepackage{txfonts}
\usepackage{mathrsfs}

\theoremstyle{plain}
\newtheorem{thm}{Theorem}[section]
\newtheorem{cor}[thm]{Corollary}
\newtheorem{lem}[thm]{Lemma}
\newtheorem{prop}[thm]{Proposition}

\newtheorem{rem}[thm]{Remark}

\def\sqr#1#2{{\vcenter{\vbox{\hrule height.#2pt
              \hbox{\vrule width.#2pt height#1pt \kern#1pt \vrule
width.#2pt}
              \hrule height.#2pt}}}}
\def\be{\begin{equation}}
\def\ee{\end{equation}}

\def\ga{{\gamma}}

\def\ep{{\epsilon}}

\def\Sp{{\mathrm {Sp}}}

\def\ga{{\gamma}}

\def\R{{\Bbb R}}

\def\<{{\langle}}
\def\>{{\rangle}}
\def\no{\noindent}

\def\bs{\bigskip}

\def\({\Big (}
\def\){\Big )}
\def\[{\Big[}
\def\]{\Big]}

\def\be{\begin{equation}}
\def\bel{\begin{equation}\label}
\def\ee{\end{equation}}
\def\bea{\begin{eqnarray}}
\def\eea{\end{eqnarray}}
\def\bt{\begin{theorem}}
\def\et{\end{theorem}}
\def\bc{\begin{corollary}}
\def\ec{\end{corollary}}
\def\bl{\begin{lemma}}
\def\el{\end{lemma}}
\def\bp{\begin{proposition}}
\def\ep{\end{proposition}}
\def\br{\begin{remark}}
\def\er{\end{remark}}
\def\ba{\begin{array}}
\def\ea{\end{array}}
\def\bd{\begin{definition}}
\def\ed{\end{definition}}

\makeatletter
   
   \@addtoreset{equation}{section}
\makeatother

\begin{document}

\title{\bf  Spectral instability of the regular $n$-gon elliptic  relative equilibrium in the planar $n$-body problem }
\author{Yuwei Ou\thanks{E-mail:ywou@sdu.edu.cn}
\quad Yunying Wang \thanks{E-mail:yunyingwyy@gmail.com}
\\ \\
School of Mathematics, Shandong University
Jinan, Shandong 250100\\
The People's Republic of China
}
\date{}
\maketitle
\begin{abstract}

 The regular $n$-gon elliptic relative equilibrium (ERE) is a Kepler homographic solution generated by the regular $n$-gon central configuration, and its linear stability depends on the eccentricity $\mathfrak{e}\in[0,1)$. While Moeckel \cite{Moe1} established the spectral instability for this solution at $\mathfrak{e}=0$ for all $n\geq3$, it remained unknown whether instability persists for $\mathfrak{e} \in (0,1)$. This paper resolves this problem: we prove that the regular $n$-gon ERE is spectral instability for all $n\geq 3$ and $\mathfrak{e} \in [0,1)$. Furthermore, we introduce the $\beta$-system (\ref{beta-system}) which related the Lagrange solution, and we developed an estimation method that, by testing the hyperbolicity of the $\beta$-system at a finite number of points alone, allows us to obtain extensive hyperbolic regions. As a corollary, for $n=3,4,5$, we uniformly demonstrate that the instability is hyperbolic (and hence stronger) for all $\mathfrak{e} \in [0,1)$.



\end{abstract}

\bs

\no{\bf AMS Subject Classification:} 70F10, 37J25, 37J46, 34L15

\bs

\no{\bf Key Words:}  linear stability, central configuration, elliptic relative equilibrium, planar $n$-body problem
\section{Introduction and main results}

For $n$ particles with masses $m_1,\cdots,m_n$, let $q_1,\cdots,q_n\in \mathbb{R}^2$ be the position vectors. Let
\bea U=\sum_{1\leq i< j\leq n} \frac{m_im_j}{\|q_i-q_j\|} \eea
be the negative potential function defined on the configuration space  $$
\Lambda=\{x=(x_1,\cdots,x_n)\in\mathbb{R}^{2n}\setminus\triangle:
\sum_{i=1}^nm_ix_i=0 \},$$ where
$\triangle=\{x\in\mathbb{R}^{2n}:\exists\, i\neq j,x_i=x_j \}$ is the
collision set.  Obviously, the orbits of the $n$ bodies satisfy the following Newton equation
\bea m_i\ddot{q}_i(t)=\frac{\partial
U}{\partial q_i}(q_1,...,q_n). \label{1.2} \eea

A planar central configuration of $n$ particles with mass center in original point is formed by $n$ position vector $a=(a_1,...,a_n)\in \mathbb{R}^{2n}$ which satisfy
\bea -\lambda \mathcal{M}a=\frac{\partial U(a)}{\partial q}, \label{CC}\eea
for constant $\lambda=U(a)/\cal{I}(a)>0$, where $\mathcal{M}=diag(m_{1},m_{1},m_{2},m_{2}\ldots,m_{n},m_{n})$,
$\cal{I}(a)=\sum m_j\|a_j\|^2$ is the moment of inertia. A planar central configuration of the $n$-body problem gives rise to a solution of
(\ref{1.2}) where each particle moves on a specific Keplerian orbit while the totality of the particles move on a homographic motion. More precisely,
the homographic solution generated by the cental configuration $a$ is
\bea\label{ERE}
x(t)=r(t)\mathfrak{R}(\theta(t))a,
\eea
where
$$
r(t)=\frac{\Omega^2/\lambda}{1+\mathfrak{e}\cos\theta(t)}, \ \ r^{2}\dot{\theta}(t)=\Omega,
$$
and $\mathfrak{R}(\theta(t))=diag(R(\theta(t)),\ldots, R(\theta(t)))$, $R(\theta(t))=\left( \begin{array}{cc} \cos \theta(t) & -\sin \theta(t)\\
\sin \theta(t) & \cos \theta(t)\end{array}\right)$, $\Omega\neq0$ is the angular momentum.
If the Keplerian orbit
is elliptic then the solution is an equilibrium in pulsating coordinates so we call this solution an elliptic
relative equilibrium (ERE for short), and a relative equilibrium in case $\mathfrak{e}=0$ (cf. \cite{MS}).

For $n=3$, there are only two kinds of central configurations, the Lagrange equilateral triangle central
configuration (see Figure (a)) and the Euler collinear central configuration (see Figure (b)).
\begin{figure}[H]
    \centering
    \begin{subfigure}[b]{0.33\textwidth}
        \includegraphics[width=\textwidth]{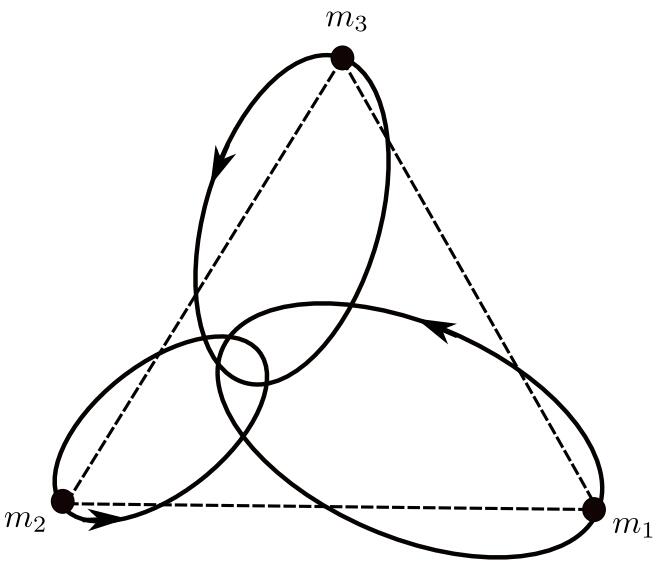}
        \caption{The Lagrange solution}
        \label{fig:sub1}
    \end{subfigure}
     \hspace{0.05\textwidth}
    \begin{subfigure}[b]{0.37\textwidth}
        \includegraphics[width=\textwidth]{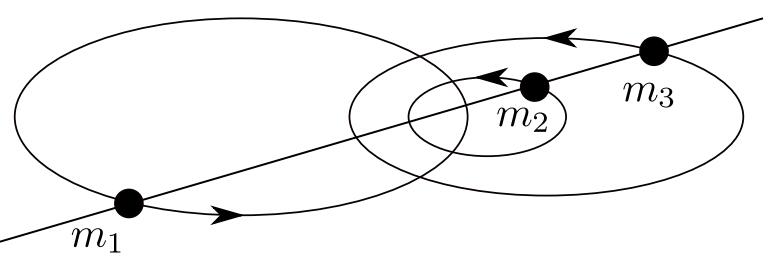}
        \caption{ The Euler solution}
        \label{fig:sub2}
    \end{subfigure}
    \label{fig:total}
\end{figure}
Lagrange solution in planar 3-body problem is found by Lagrange in 1772, which forms a equilateral triangle all the time. The study on the stability of Lagrange solution has a long history. From Gascheau \cite{G43} in 1843 for circle Lagrange solution to Danby\ \cite{D64} in 1964 for elliptic case, the stability of Lagrange solution can be described by two parameters, mass parameter
\bea
\beta_{L}=\frac{27(m_1m_2+m_1m_3+m_2m_3)}{(m_1+m_2+m_3)^2} \in [0,9]\nonumber
\eea
and eccentricity $\mathfrak{e}\in [0,1)$. Euler solution is another type long-historical ERE in planar 3-body problem discovered by Euler in 1767, which keeps collinear all the time.
The stability of Euler solution can be described by
\bea
\beta_{E}=\frac{m_1(3\mu^2+3\mu+1)+m_3\mu^2(\mu^2+3\mu+3)}{\mu^2+m_2[(\mu+1)^2(\mu^2+1)-\mu^2]} \in [0,7]\nonumber
\eea
and eccentricity $\mathfrak{e}\in [0,1)$, where $\mu$ is the unique positive solution of the Euler quintic polynomial equation, decided by the Euler configuration, the details can be found in \cite{ZL}.
In \cite{MSS1, MSS2}, Mart\'{\i}nez, Sam\`{a} and Sim\'{o} obtained the complete bifurcation diagrams numerically to describe the parameter regions of Lagrange solution and Euler solution and beautiful figures were drawn there for the full $(\beta,e)$ range. Hu and Sun firstly introduced Maslov-type index into the study of stability for Lagrange solution in \cite{HS}. In \cite{HLS}, Hu, Long and Sun further developed this method and gave a complete analytically description of the stability bifurcation diagrams of \cite{MSS1} for Lagrange solution. Later in \cite{BJP}, Barutello, Jadanza, and Portaluri studied the linear stability of Lagrange circular orbits with
$\alpha$-potentials using index theory. Further, Hu, Wang and the first author in \cite{HOW15} and \cite{HOW19} developed the trace formula for Hamiltonian system and firstly used it to study the stability of elliptic relative equilibria quantitatively. They estimated the stable region for elliptic Lagrange solution and elliptic-hyperbolic region for Euler solution with eccentricity $\mathfrak{e}$ less than some $\mathfrak{e}_0<0.34$.
Zhou and Long in \cite{ZL, ZL2} study the stability of elliptic Euler solution and Euler-Moulton solutions with the similar method. Hu and the first author in \cite{HO} develop the collision index to study the stability bifurcation diagram of Euler solution for the limit case (i.e $\mathfrak{e}\rightarrow1$). Recently, Hu, Sun and the first author in \cite{HOS} provide explicit stability estimates for the Lagrange, Euler
solutions over the full range of eccentricity.

For $n\geq 4$, it is difficult to find all central
configurations. The well know examples include the regular $n$-gon ($n$ equal masses $m_{k}$ placed at the vertices of a regular $n$-gon, see Figure (c))
and the regular $(1+n)$-gon (the $n$-gon configuration with a central body of arbitrary mass $m>0$, see Figure (d)).
Without loss of generality, we set $m_{k}=1$, for $k\in\{1,...,n\}$.
\begin{figure}[H]
    \centering
    \begin{subfigure}[b]{0.32\textwidth}
        \includegraphics[width=\textwidth]{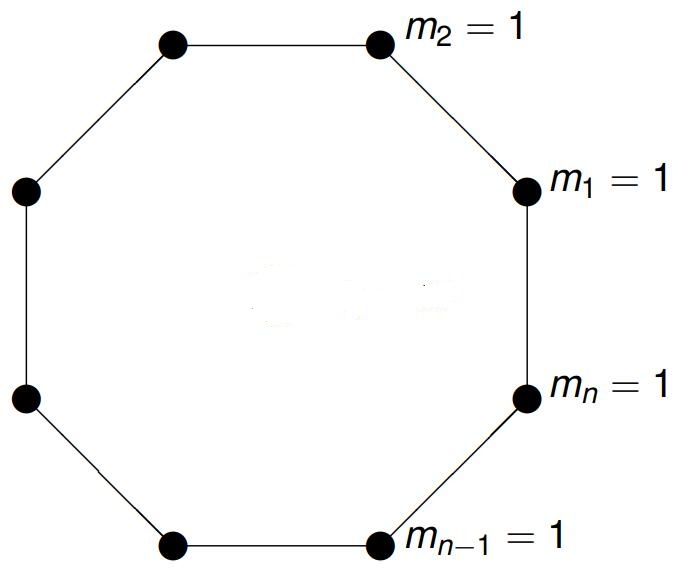}
        \caption{regular $n$-gon solution}
        \label{fig:sub3}
    \end{subfigure}
     \hspace{0.05\textwidth}
    \begin{subfigure}[b]{0.32\textwidth}
        \includegraphics[width=\textwidth]{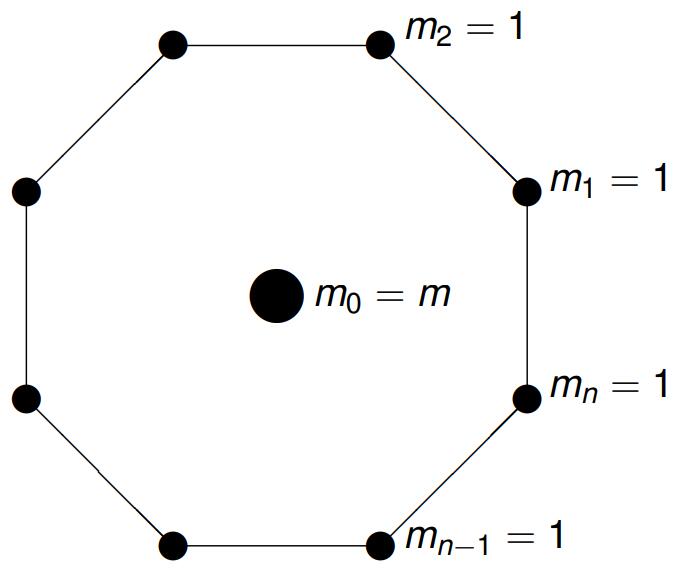}
        \caption{regular $(1+n)$-gon solution}
        \label{fig:sub4}
    \end{subfigure}
    \label{fig:total}
\end{figure}
The linear stability of the regular $(1+n)$-gon ERE depends on the mass $m$ of the central body and the eccentricity $\mathfrak{e}\in[0,1)$, while the linear stability
of the regular $n$-gon only depends on the eccentricity $\mathfrak{e}\in[0,1)$.
There have existed many works which studied the linear stability of the regular $(1+n)$-gon ERE and the regular $n$-gon ERE,
for the case with $\mathfrak{e}=0$. As far as we know, the regular $(1+n)$-gon was first started by Maxwell in his study on the stability of
Saturn's rings (cf. \cite{Max1,Max2}). Moeckel \cite{Moe2} proved that the regular $(1+n)$-gon is linearly stable for sufficiently
large $m$ only when $n\geq7$. For $n\geq 7$, Roberts found a value $h_n$ which is proportional to $n^3$, and the
regular $(1+n)$-gon is stable if and only if $m>h_n$ (cf. \cite{Rob1}). For the regular $n$-gon, Moeckel \cite{Moe1} shows that it
is spectral instability for any $n\geq 3$. Roberts \cite{Rob2} gives a simple proof of the instability of the regular $n$-gon for $n\geq7$.
For other related works, please refer to \cite{VK} and reference therein.

To the best of our knowledge, there are few results of the elliptic relative equilibria (i.e., the case with $\mathfrak{e}>0$) for the regular $(1+n)$-gon and
the regular $n$-gon central configurations. Recently, for the case $\mathfrak{e}>0$, the linear stability of the regular $(1+n)$-gon was first studied by Hu, Long and the first author in \cite{HLO}, they showed that for $n\geq 8$ and any eccentricity $\mathfrak{e}\in[0,1)$, the regular $(1+n)$-gon ERE is linearly stable when the central mass $m$ is large enough. Later in \cite{OS}, Sun and the first author also proved the regular $(1+7)$-gon is linearly stable when the central mass $m$ is large enough. This extends Moeckel's stability results for the regular $(1+n)$-gon ERE to arbitrary eccentricities $\mathfrak{e}\in[0,1)$. However, for the regular $n$-gon ERE, except the unstable results for $n=3, 4$ in \cite{HLS, HO}, we are not aware of any unstable result for $n\geq5$ and $\mathfrak{e}>0$.

In this paper, we study the spectral instability of the regular $n$-gon ERE for any $n\geq3$ and $\mathfrak{e}\in[0,1)$. Based on Meyer and Schmidt coordinate, we change the linear Hamiltonian system of ERE to a nice form by some symplectic transformation, see Theorem \ref{red-ERE}. In this nice form, we can easily see how the symmetry affects the system. Moreover, in Theorem \ref{redu-n-gon} we give the explicit expression of the regular $n$-gon system in the new coordinate, we can decompose this system to three subsystems, one is associated to the translation symmetry, the second is associated to the dilation and rotation symmetries of the system which is just the linear part of the Kepler orbits, and the third is the essential part. Based on this decomposition, we give the stability analysis of the regular $n$-gon system and prove the spectral instability of the regular $n$-gon ERE for all $n \geq 3$ and any eccentricity $\mathfrak{e} \in [0,1)$. This resolves a natural question arising from Moeckel's result \cite{Moe1}, which showed instability at $\mathfrak{e}=0$. Moreover, we introduce the $\beta$-system (\ref{beta-system}) which related the Lagrange solution, and we developed an estimation method that, by testing the hyperbolicity of the $\beta$-system at a finite number of points alone, allows us to obtain extensive hyperbolic regions. As a corollary,
for $n=3,4,5$, we demonstrate that the instability is hyperbolic (and hence stronger) for all $\mathfrak{e} \in [0,1)$. These findings extend and strengthen the understanding of linear stability for Kepler homographic solutions based on regular $n$-gon central configurations.

\textbf{Notations}. Following notations will be adopted throughout the paper.
\begin{itemize}
\item  We denote the real number set, complex number set, the non-negative integer
set ans the unit circle by $\mathbb{R}$, $\mathbb{C}$, $\mathbb{N}$ and $\mathbb{U}$ respectively.
\item Let $I_j$ be the identity matrix on $\mathbb{R}^j$ and
$J_{2j}=\left( \begin{array}{cccc}0_j& -I_j \\
                                  I_j& 0_j \end{array}\right)$, $\mathbb{J}_n=diag(J_2,...,J_2)_{2n\times2n}$. For simplicity, sometimes we omit the sub-indices of $I$ and $\mathbb{J}$,
but can be easily found out through the context.
\item Given a function $f: \mathbb{R}^{k} \to \mathbb{R}$ and matrix $C$, $\nabla f$ represents the gradient of $f$ with respect to the Euclidean inner product expressed as a column vector and
$D^2 f$ denote the Hessian of $f$, $C^{T}$ denote the transpose matrix of $C$.
\item We denote by $GL(\R^{2n})$ the invertible matrix group and $\text{Sym}(\R^{2n})$ be the set of symmetric matrix in $\mathbb{R}^{2n}$,
$$   \Sp(2n)=\{M\in GL(\R^{2n}),  M^TJM=J\}  $$
be the symplectic group.
\item As in \cite{Lon4}, for $M_1=\left( \begin{array}{cccc}A_1& A_2 \\
A_3 & A_4 \end{array}\right)$,  $M_2=\left( \begin{array}{cccc}B_1& B_2 \\
B_3 & B_4 \end{array}\right)$, the symplectic sum $\diamond$ is defined by
\bea  M_1\diamond M_2=\left( \begin{array}{cccccccc}A_1& 0 &A_2 & 0 \\
                                                     0 & B_1& 0  & B_2 \\
                                                    A_3& 0 &A_4 & 0 \\
                                                     0 & B_3& 0  & B_4 \end{array}\right).\nonumber   \eea
\item In what follows we write $A\geq B$ for two linear symmetric
operators $A$ and $B$, if $A-B\geq 0$, i.e. $A-B$ possesses no negative eigenvalues and
write $A>B$, if $A-B>0$, i.e. $A-B$ possesses only positive eigenvalues.
\end{itemize}
Then we have the following theorems.
\begin{thm}\label{red-ERE}
For the planar ERE which produced by the planar central configuration $a$, after some symplectic transformations, the fundamental solution satisfies the following equation,
\bea
\zeta'_{\mathfrak{e}}(\theta)=J_{4n}B(\theta)\zeta_{\mathfrak{e}}(\theta),\, \zeta_{\mathfrak{e}}(0)=I_{4n}.\label{EREH}
\eea
where
\bea B(\theta)=\left( \begin{array}{cccc} I_{2n} & -\mathbb{J}_{n} \\
\mathbb{J}_{n} & I_{2n}-\frac{I_{2n}+\frac{1}{\lambda}A^{T}D^2U(a)A}{1+\mathfrak{e}\cos\theta}
\end{array}\right),\quad \theta\in[0,2\pi], \label{msf} \eea
where $\mathfrak{e}$ is the eccentricity, $\lambda=U(a)/I(a)$ and
$A\in GL(\R^{2n})$ satisfies
\bea
\mathbb{J}_{n}A = A\mathbb{J}_{n},\ \ A^T\mathcal{M}A = I_{2n},\label{A}
\eea
 which is introduced by Meyer and Schmidt in \cite{MS}.
\end{thm}
\begin{rem}
This form (\ref{msf}) is not explicitly given in \cite{MS}, it is first mentioned in \cite{HO}, but the detailed proof is lacking. For the convenience of readers in future applications, we provide a concrete proof in Section $2$. Also, in Section \ref{Se2}, based on this nice form, we can easily see how the symmetry contribute the characteristic multiplier $+1$ of the monodromy matrix $\zeta_{\mathfrak{e}}(2\pi)$ and affects the system.
\end{rem}

\begin{thm}\label{redu-n-gon}
We construct $A$ by (\ref{matrix A}) in Section \ref{Se3}, then the linear Hamiltonian system of the regular $n$-gon system in Theorem \ref{red-ERE}
is given by
\bea  \zeta'_{\mathfrak{e}}(\theta)=J_{4n}B(\theta)\zeta_{\mathfrak{e}}(\theta), \zeta_{\mathfrak{e}}(0)=I_{4n}. \nonumber \eea
with $B(\theta)=B_{1}(\theta)\diamond B_{2}(\theta)\diamond B_{3}(\theta)$, where
$$
B_{1}(\theta)=\left( \begin{array}{cccc} I_{2} & -\mathbb{J}_{1} \\
\mathbb{J}_{1} & I_{2}-\frac{I_{2}}{1+\mathfrak{e}\cos\theta}
\end{array}\right),\ \
B_{2}(\theta)=\left( \begin{array}{cccc} I_{2} & -\mathbb{J}_{1} \\
\mathbb{J}_{1} & I_{2}-\frac{I_{2}+\mathcal{U}_{0}
}{1+\mathfrak{e}\cos\theta}
\end{array}\right),\\
$$
\bea
B_{3}(\theta)=\mathcal{B}_1(\theta)\diamond\cdots\diamond \mathcal{B}_{[\frac{n}{2}]}(\theta). \nonumber\eea
$B_{1}(\theta)$ is associated to the translation symmetry, $B_{2}(\theta)$ associated to the dilation and rotation symmetries
and $B_{3}(t)$ corresponds to the core part of the linearized system. Moreover we have
\bea
\mathcal{B}_{l}(\theta)=\left( \begin{array}{cccc} I & -\mathbb{J}\\
\mathbb{J} & I-\frac{I+\mathcal{U}_{l}
}{1+\mathfrak{e}\cos\theta}  \end{array}\right), \, l=1,\cdots,[\frac{n}{2}]. \nonumber \eea
where $\mathcal{U}_{l}$ is given by Propostion \ref{lem-redu}. Here we omit the sub-indices of $I$ and $\mathbb{J}$, which are chosen to have the same dimensions as those of $\mathcal{U}_l$.
\end{thm}

\begin{thm}\label{unstable}
For $n\geq3$, the subsystem $\mathcal{B}_{1}$ of the regular $n$-gon ERE is hyperbolic for any eccentricity $\mathfrak{e}\in[0,1)$, consequently, the regular $n$-gon ERE is spectral instablity for any $\mathfrak{e}\in [0,1)$.
\end{thm}
In particular, for $n = 3,4,5$ and $\mathfrak{e} = 0$, Moeckel in \cite{Moe1} further established that the regular $n$-gon ERE is hyperbolic. It is therefore natural to extend these results to arbitrary eccentricity $\mathfrak{e} \in (0,1)$. To achieve this goal, in Section $5$, we further develop an analytical technique and estimate the hyperbolic region for the following
$\beta$-system,
$$
\gamma_{\beta,\mathfrak{e}}^{\prime} (\theta) =J \mathcal{B}_{\beta,\mathfrak{e}} (\theta) \gamma_{\beta,\mathfrak{e}},\ \ \gamma_{\beta,\mathfrak{e}} (0) =I_{4} .
$$
where
$$
\mathcal{B}_{\beta,\mathfrak{e}} (\theta) =\left (\begin{array}{cc}
I_{2} & -J_{2} \\
J_{2} & I_{2}-\frac{\mathcal{R}_{\beta}}{1+\mathfrak{e} \cos \theta}
\end{array}\right),\ \ \mathcal{R}_{\beta}=\frac{3}{2}I_{2}+\beta\left(
                                                                   \begin{array}{cc}
                                                                     -1 & 0 \\
                                                                     0 & 1 \\
                                                                   \end{array}
                                                                 \right),\ \ \beta\geq0.
$$
Although numerical evidence suggests that $\gamma_{\beta,\mathfrak{e}}(2\pi)$ exhibits persistent hyperbolicity over extensive regions, such methods bear a fundamental limitation: they are theoretically incapable of verifying hyperbolicity at infinitely many points. Also, the errors of the numerical methods is hard to control when the eccentricity is closed to one, because the linear equation has singularity at $\mathfrak{e}=1$. Based on the reason, we further obtain the following Theorem \ref{hyper}, which tells us that if we know
the hyperbolicity of system (\ref{beta-system}) for some fixed points $(\beta_{0}, \mathfrak{e}_{0})$, then we can obtain a large hyperbolic region. By carefully selecting appropriate parameters $(\beta_{0}, \mathfrak{e}_{0})$, we only need to check the hyperbolicity of the $\beta$-system at this finite points $(\beta_{0}, \mathfrak{e}_{0})$.
\begin{thm}\label{hyper}
Take finite set $\mathcal{K}=\{(1.36,0.0), (1.36,0.1),(1.36,0.2),\ldots,(1.36,0.9)\}$, numerically, one can check $\gamma_{\beta,\mathfrak{e}}(2\pi)$ is hyperbolic for $(\beta_{0},\mathfrak{e}_{0})\in\mathcal{K}$. Then $\gamma_{\beta,\mathfrak{e}}(2\pi)$ is hyperbolic in region $(\beta,\mathfrak{e})\in\mathfrak{U}$,
where
$$
\mathfrak{U}=\bigcup_{(\beta_{0},\mathfrak{e}_{0})\in\mathcal{K}}\big(\mathfrak{U}_{1}(\beta_{0},\mathfrak{e}_{0})\cup\mathfrak{U}_{2}(\beta_{0},\mathfrak{e}_{0})\big)
$$
with
\bea
\mathfrak{U}_{1}(\beta_{0},\mathfrak{e}_{0})=\big\{(\beta,\mathfrak{e}) | 0\leq\beta<\frac{1+\mathfrak{e}}{1+3\mathfrak{e}-2\mathfrak{e}_{0}}\beta_{0},\ \ \mathfrak{e}_{0}\leq\mathfrak{e}  \big\},\nonumber\\
\mathfrak{U}_{2}(\beta_{0},\mathfrak{e}_{0})=\big\{(\beta,\mathfrak{e}) | 0\leq\beta<\frac{1-\mathfrak{e}}{1-3\mathfrak{e}+2\mathfrak{e}_{0}}\beta_{0},\ \ \mathfrak{e}_{0}\geq\mathfrak{e}  \big\}.\nonumber
\eea
\begin{figure}[H]
    \centering
    \includegraphics[width=10cm]{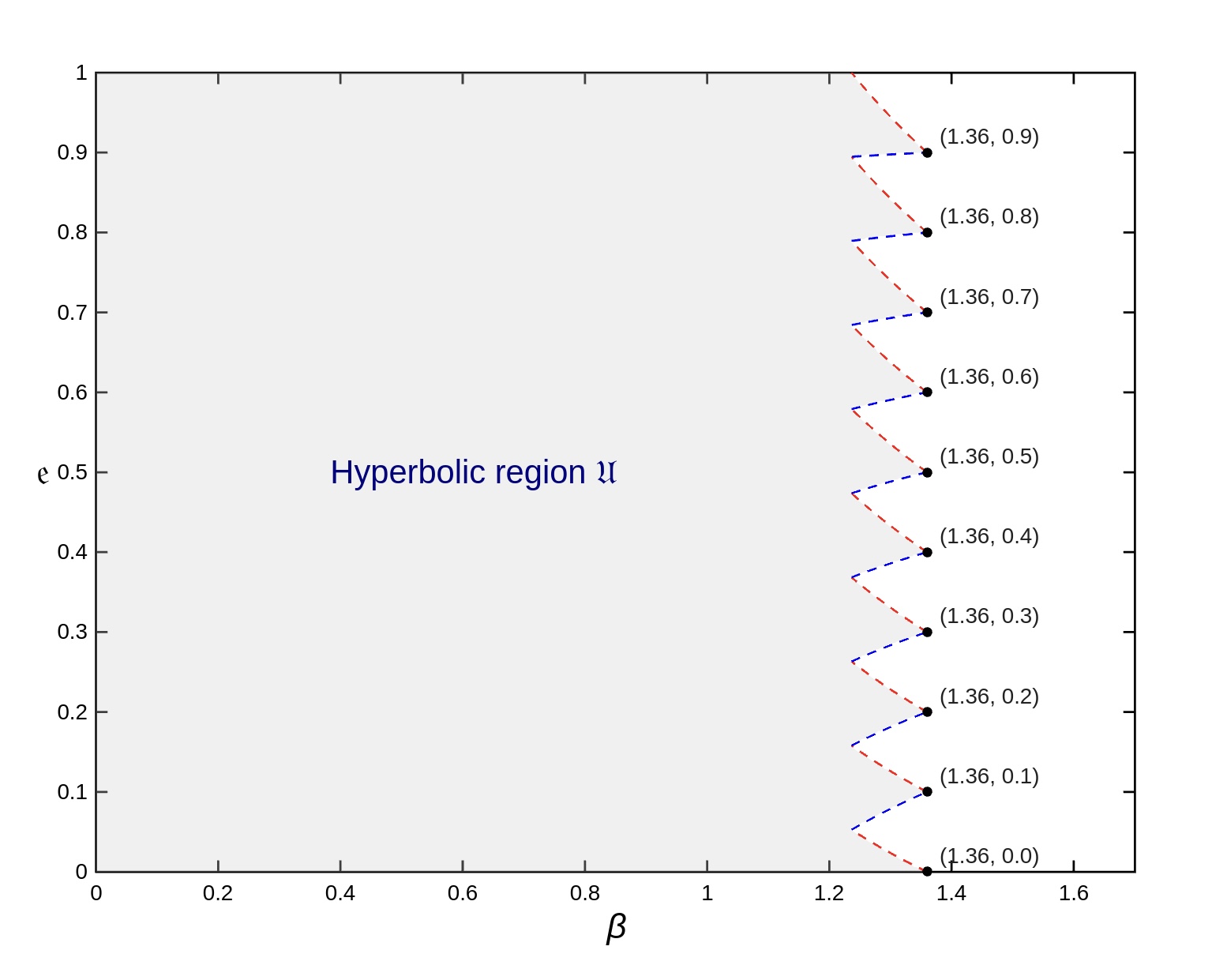}
    \caption{Hyperbolic region $\mathfrak{U}$ (gray)}
    \label{fig:tg-}
\end{figure}
\end{thm}
Based on Theorem \ref{hyper}, we will provide a unified proof of the following corollary.
\begin{cor}\label{hyperbolicity}
For $n=3, 4, 5$ and any $\mathfrak{e}\in [0,1)$, the regular $n$-gon ERE is hyperbolic.
\end{cor}
\begin{rem}
In \cite{HLS} and \cite{HO}, the authors has already established the hyperbolicity of the regular $n$-gon ERE for $n = 3$ and $n = 4$, respectively. However, for $n = 5$, the problem remains open and presents greater complexity than the $n = 3, 4$ cases. In fact, the hyperbolicity of the regular $3$-gon is equivalent to the hyperbolicity of $\gamma_{\beta,\mathfrak{e}}(2\pi)$ in region $(\beta,\mathfrak{e})\in\{0\}\times[0,1)$, this non-trivial analytic result was first obtained by \cite{HLS}. In \cite{O}, the first author expands the hyperbolic region for $(\beta,\mathfrak{e})\in[0,0.5)\times[0,1)$. Further, in \cite{HO}, Hu and the first author further extended the hyperbolic region for $(\beta,\mathfrak{e})\in[0,0.7237)\times[0,1)$
, this implies the hyerpbolicity of the regular $4$-gon ERE for any eccentricity, but this result is not enough to obtain the hyperbolicity of the
regular $5$-gon ERE, hence it is still an unsolved problem. In order to get the hyperbolicity of the regular $5$-gon ERE, we will later see in Section $5$, that we need to establish hyperbolicity for $(\beta,\mathfrak{e})\in[0, 1.1459)\times[0,1)\subset \mathfrak{U}$.
\end{rem}
\begin{rem}
Notably, when $n\geq6$, as shown by Moeckel \cite{Moe1} the regular $n$-gon is not hyperbolic even at eccentricity $\mathfrak{e}=0$. Some subsystems $\mathcal{B}_{l}$ within its essential subsystem decomposition are stable. Consequently, for regular regular $n$-gon with $n\geq6$, one can only conclude spectral instability but not the hyperbolic for
any $\mathfrak{e}\in[0,1)$.
\end{rem}

This paper is organized as follows. In Section $2$, we explain the reduction of ERE and prove Theorem 1.1.
In Section $3$, we prove Theorem $1.2$, that is give the expression of the regular $n$-gon ERE under the reduction coordinate of Theorem $1.1$.
In Section $4$, we study the spectral unstable of the regular $n$-gon ERE and prove
Theorem \ref{unstable}. In Section $5$, we will prove Theorem \ref{hyper}, where we introduce the $\beta$-system and developed
an estimation method that, by testing the hyperbolicity at a finite number of points alone, allows us to
determine the hyperbolicity of this system for arbitrary eccentricities. As a corollary, we provide a unified proof of the hyperbolicity of the regular $n$-gon for $n=3,4,5$.

\section{Reduction of the elliptic relative equilibria}\label{Se2}
Considering $n$ particles with masses $m_1,...,m_n$, let $Q=(q_1,...,q_n)\in (\R^2)^n$ be
the position vector, and $P=(p_1,...,p_n)\in (\R^2)^n$ be the momentum vector. Denote by $d_{ij}=||q_{i}-q_{j}||$,
the Hamiltonian function has the form
\bea
H(P,Q)=\sum_{j=1}^n\frac{||p_{j}||^2}{2m_{j}}-U(Q),\ \ U(Q)=\sum_{1\leq j<i\leq n}\frac{m_{j}m_{i}}{d_{ji}}.\nonumber
\eea
We denote by $\mathbb{J}_n=diag(J_2,...,J_2)_{2n\times2n} $ and  $\mathcal{M}=diag (m_1,m_1,m_2,m_2,...,m_n, m_n)_{2n\times2n}$.
Then the corresponding fundamental solution $\ga$ of periodic solution $(P(t), Q(t))$ is given by
\bea \dot{\ga}(t)=J_{4n}D^2H(P(t),Q(t))\ga(t),\,\ \ga(0)=I_{4n}. \label{ga}  \eea
where
\bea
D^2H(P(t),Q(t))=\left(\begin{array}{cccc} \mathcal{M}^{-1} & O \\
O & -D^2U(Q(t))
\end{array}\right),\nonumber
\eea
$O$ is the zero matrix.

The periodic solution $Q(t)$ with minimal periodic $\mathcal{T}$ is called spectrally stable if all eigenvalues of $\gamma(\mathcal{T})$ belong to the unit
circle $\mathbb{U}$ of the complex plane. $Q(t)$ is called linearly stable if $\gamma(\mathcal{T})$ is spectrally stable and semi-simple. While $Q(t)$ is called hyperbolic if no eigenvalues of $\gamma(\mathcal{T})$ are on $\mathbb{U}$. For the ERE, from Meyer and Schmidt \cite{MS}, there are two four-dimensional invariant symplectic subspaces, $E_1$ and $E_2$, and they are associated to the translation symmetry, dilation and rotation symmetry of the system. In other words, there is a symplectic coordinate system in which the linearized system of the planar n-body problem decouples into three subsystems on $E_1, E_2$ and $E_3=(E_1\cup E_2)^{\bot}$, where $\bot$ denotes the symplectic orthogonal complement.
Due to the symmetry, $\gamma(\mathcal{T})$ restricted to $E_1\cup E_2$ give the characteristic multiplier $+1$ with multiplicity of $8$, hence the ERE is called spectral stable (linearly stable, hyperbolic, resp.) if the monodromy matrix $\gamma(\mathcal{T})$ restricted to $E_3$, $\gamma(\mathcal{T})|_{E_3}$ is spectral stable
(linearly stable, hyperbolic, resp.).


In the following, we analysis the linear Hamiltonian sytem (\ref{ga}) directly, by using
the linear symplectic transformations, we give a nice form of the linear Hamiltonian system of ERE, see Theorem \ref{red-ERE},
in this nice form, we can easily see how the symmetry affects this system.

For the elliptic relative equilibrium (\ref{ERE})
$$
x(t)=r(t)\mathfrak{R}(\theta(t))a
$$
which generated by the cental configuration $a$, where $ r(t)=\frac{\Omega^2/\lambda}{1+e\cos\theta(t)}$. Consider the linear Hamiltonian system (\ref{ga}) at $(P(t), Q(t))$ with $P(t)=\mathcal{M}\dot{Q}(t), Q(t)=x(t)$.
Based on the relation $r^{2}\dot{\theta}(t)=\Omega\neq0$, we first change variable $t$ to the true anomaly $\theta$.
Without loss of generality, we assume $\Omega>0$ and the initial condition $\theta(0)=0$, then $\theta(\mathcal{T})=2\pi$, where $\mathcal{T}$ is the minimal period of $x(t)$.

Now define $\tilde{\gamma}(\theta)=\gamma(t(\theta))$, $\tilde{r}(\theta)=r(t(\theta))=\frac{\Omega^2/\lambda}{1+\mathfrak{e}\cos\theta}$,
the derivative corresponding to $\theta$ is denoted by $'$, then direct computation shows that
\bea
\tilde{\gamma}'(\theta)=J_{4n}\tilde{B}(\theta)\tilde{\gamma}(\theta),\ \ \tilde{\gamma}(0)=I_{4n},\label{timechange}
\eea
where
$$
\tilde{B}(\theta)=\left(\begin{array}{cccc} \frac{\tilde{r}^2}{\Omega}\mathcal{M}^{-1} & O \\
O & -\frac{\tilde{r}^2}{\Omega}D^2U(Q(t(\theta)))
\end{array}\right).
$$
We can further simplify $\tilde{B}(\theta)$, since for any $z\in\mathbb{R}^{2n}$, $U(\tilde{r}\mathfrak{R}(\theta)z)=U(z)/\tilde{r}$, we have
\bea
\frac{\partial^{2}U(\tilde{r}\mathfrak{R}(\theta)z)}{\partial z^2}\Big{|}_{z=a}=\frac{D^2U(a)}{r},\label{drU1}
\eea
on the other hand, direct computation shows that
\bea
\frac{\partial^{2}U(\tilde{r}\mathfrak{R}(\theta)z)}{\partial z^2}\Big{|}_{z=a}=\tilde{r}^2\mathfrak{R}^{T}(\theta)D^{2}U(\tilde{r}\mathfrak{R}(\theta)a)\mathfrak{R}(\theta),\label{drU2}
\eea
compare equalities (\ref{drU1}) and (\ref{drU2}), for $Q(t(\theta))=r\mathfrak{R}(\theta)a$, we have
$$
D^{2}U(Q(t(\theta))=\frac{1}{\tilde{r}^3}\mathfrak{R}^{-T}(\theta)D^2U(a)\mathfrak{R}^{-1}(\theta),
$$
hence
\bea \label{B1}
\tilde{B}(\theta)=\left(\begin{array}{cccc} \frac{\tilde{r}^2}{\Omega}\mathcal{M}^{-1} & O \\
O & -\frac{\mathfrak{R}^{-T}(\theta)D^2U(a)\mathfrak{R}^{-1}(\theta)}{\tilde{r}\Omega}
\end{array}\right).\nonumber
\eea
In order to get Theorem \ref{red-ERE}, we introduce the following lemma to make further transformation,
\begin{lem}\label{symptran}
Let $S(\theta)$ be a path of symplectic matrices depending on $\theta$, if $\hat{\gamma}(\theta)=S^{-1}(\theta)\tilde{\gamma}(\theta)S(0)$, where
$\tilde{\gamma}(\theta)$ is given by (\ref{timechange}) then
$$
\hat{\gamma}'(\theta)=J\Psi_{S}(\tilde{B})\hat{\gamma}(\theta),\ \ \hat{\gamma}(0)=I_{4n}
$$
where $\Psi_{S}(\tilde{B})=S^{T}(\theta)\tilde{B}(\theta)S(\theta)+JS^{-1}(\theta)S'(\theta)$.
\end{lem}
\begin{proof}
The proof is obtained by direct calculation, and we leave it to the readers.
\end{proof}
Now we take
$$
S(\theta)=\left(\begin{array}{cccc} \frac{\sqrt{\Omega}}{\tilde{r}}A^{-T} & O \\
O & -\frac{\tilde{r}}{\sqrt{\Omega}}A
\end{array}\right)\left(\begin{array}{cccc} I & \frac{\tilde{r}'}{\tilde{r}}I \\
O & I
\end{array}\right),
$$
where $A\in GL(\R^{2n})$ satisfies
$$
\mathbb{J}_{n}A = A\mathbb{J}_{n},\ \ A^T\mathcal{M}A = I_{2n}.
$$
Let $\hat{\gamma}(\theta)=S^{-1}(\theta)\tilde{\gamma}(\theta)S(0)$, from Lemma \ref{symptran}, we have
\bea
\hat{\gamma}'(\theta)=J\Psi_{S}(\tilde{B})\hat{\gamma}(\theta),\ \ \hat{\gamma}(0)=I_{4n}.\nonumber
\eea
where
\bea
\Psi_{S}(\tilde{B})=\left( \begin{array}{cccc} I_{2n} & O \\
O & [(\frac{\tilde{r}'(\theta)}{r(\theta)})'-(\frac{r'(\theta)}{r(\theta)})^2]I_{2n}-\frac{\tilde{r}(\theta)}{\Omega^2}A^T\mathfrak{R}^{-T}(\theta)D^2U(a)\mathfrak{R}^{-1}(\theta)A\nonumber
\end{array}\right).
\eea
Since $\mathbb{J}_{n}A = A\mathbb{J}_{n}$, we have $A\mathfrak{R}(\theta)=\mathfrak{R}(\theta)A$, then combine with following relations
$$
\tilde{r}(\theta)=\frac{\Omega^2/\lambda}{1+\mathfrak{e}\cos\theta},\ \ \tilde{r}'(\theta)=\frac{\mathfrak{e}\sin\theta}{1+\mathfrak{e}\cos\theta}\tilde{r}(\theta),
$$
direct computation shows that
\bea
\Psi_{S}(\tilde{B})=\left( \begin{array}{cccc} I_{2n} & O \\
O & I_{2n}-\frac{I_{2n}+\frac{1}{\lambda}\mathfrak{R}^{-T}(\theta)A^{-1}\mathcal{M}^{-1}D^2U(a)A\mathfrak{R}^{-1}(\theta)}{1+\mathfrak{e}\cos\theta}\nonumber
\end{array}\right).
\eea
Define $\zeta_{\mathfrak{e}}(\theta)=\left(
                            \begin{array}{cc}
                              \mathfrak{R}^{-1}(\theta) & O \\
                              O & \mathfrak{R}^{-1}(\theta)\\
                            \end{array}
                          \right)\hat{\gamma}(\theta)
$ and combine with $A^T\mathcal{M}A = I_{2n}$, then we have
\bea
\zeta'_{\mathfrak{e}}(\theta)=J_{4n}\left( \begin{array}{cccc} I_{2n} & -\mathbb{J}_{n} \\
\mathbb{J}_{n} & I_{2n}-\frac{I_{2n}+\frac{1}{\lambda}A^{T}D^2U(a)A}{1+\mathfrak{e}\cos\theta}
\end{array}\right)\zeta_{\mathfrak{e}}(\theta),\ \ \zeta_{\mathfrak{e}}(0)=I_{2n}.\nonumber
\eea
This completes the proof of Theorem 1.1.
\begin{rem} From the periodicity of $S(\theta)$, that is $S(2\pi)=S(0)$, we have $\tilde{\gamma}(2\pi)=S(0)\zeta_{\mathfrak{e}}(2\pi)S^{-1}(0)$, hence $\tilde{\gamma}(2\pi)$
 and $\zeta_{\mathfrak{e}}(2\pi)$ have the same spectrum, the stability of ERE is also determined by $\zeta_{\mathfrak{e}}(2\pi)$.
\end{rem}
Based on the nice form of linear Hamiltonian system in Theorem 1.1, we
can easily see how the symmetry contribute the characteristic multiplier $+1$ of the monodromy matrix $\zeta_{\mathfrak{e}}(2\pi)$ and affects the system.

\textbf{Translation symmetry:} Let $e_{1}=(1,0,1,0,\ldots,1,0)^{T}$, for the planar central configuration $a$, we have
$$
U(a+se_{1})=U(a), \nabla U(a+se_{1})=\nabla U(a),\ \ \forall s\in \mathbb{R}.
$$
then we have $\frac{\partial}{\partial s}\nabla U(a+se_{1})|_{s=0}=0$, which implies
\bea
D^{2}U(a)e_{1}=0. \label{eig1}
\eea
Similar, for $\mathbb{J}_{n}e_{1}=(0,1,0,1,\cdots,0,1)$, we have
\bea
D^{2}U(a)\mathbb{J}_{n}e_{1}=0. \label{eig2}
\eea
hence $e_{1}, \mathbb{J}_{n}e_{1}$ are two eigenvectors of $\frac{1}{\lambda}\mathcal{M}^{-1}D^2U(a)$ corresponding to eigenvalue $0$.

\textbf{Rotation symmetry:} Let $\mathfrak{R}(s)=diag(R(s),\ldots, R(s)), s\in \mathbb{R}$, if $a$ is a planar central configuration, then
$\mathfrak{R}(s)a$ is also a planar central configuration, from the of central configuration equation (\ref{CC}), we have
$$
\nabla U(\mathfrak{R}(s)a)=-\lambda\mathcal{M}\mathfrak{R}(s)a,
$$
Taking the derivative of $s$ at $0$ on both sides and from the fact $\frac{\partial\mathfrak{R}(s)}{\partial s}\big{|}_{s=0}=\mathbb{J}_{n}$, we get
\bea
\frac{1}{\lambda}\mathcal{M}^{-1}D^2U(a)\mathbb{J}_{n}a=-\mathbb{J}_{n}a. \label{eig3}
\eea
hence $\mathbb{J}_{n}a$ is an eigenvector of $\frac{1}{\lambda}\mathcal{M}^{-1}D^2U(a)$ corresponding to eigenvalue $-1$.

\textbf{Dilation symmetry:} For planar central configuration $a$ and any $s\in \mathbb{R}$, we have $\nabla U(sa)=\frac{1}{s^2}\nabla U(a)$.
Taking the derivative of $s$ at $1$ on both sides, we obtain
$$
D^2U(a)a=-2\nabla U(a),
$$
from the central configuration equation (\ref{CC}), we have
$$
D^2U(a)a=2\lambda\mathcal{M}a,
$$
hence
\bea
\frac{1}{\lambda}\mathcal{M}^{-1}D^2U(a)a=2a, \label{eig4}
\eea
$a$ is an eigenvector of $\frac{1}{\lambda}\mathcal{M}^{-1}D^2U(a)$ corresponding to eigenvalue $2$.

In order to reduce the symmetry, the matrix $A$ which satisfies (\ref{A}) is denoted by
\bea
A=\left( \begin{array}{cccc}
A_{11} & A_{12} & \cdots & A_{1n} \\
A_{21} & A_{22} & \cdots & A_{2n}\\
\cdots & \cdots & \cdots & \cdots\\
A_{n1} & A_{n2} & \cdots & A_{nn}\end{array}\right), \nonumber
\eea
the first four columns of matrix $A$ is constructed by
\bea
\left( \begin{array}{c}A_{11} \\ A_{21} \\\vdots \\A_{n1}\end{array}\right)
=\frac{1}{\sqrt{m}}(e_{1}, \mathbb{J}_{n}e_{1}), \ \
\left( \begin{array}{c}A_{12} \\ A_{22} \\\vdots \\A_{n2}\end{array}\right)
=\frac{1}{\sqrt{\cal{I}(a)}}(a,\mathbb{J}_{n}a),\nonumber
\eea
then from (\ref{eig1}), (\ref{eig2}) of the translation symmetry, (\ref{eig3}) of rotation symmetry and (\ref{eig4}) of dilation symmetry, we have
\bea
\frac{1}{\lambda}A^{-1}\mathcal{M}^{-1}D^2U(a)A=diag(O_{2}, \mathcal{U}_{0}, \mathcal{U}),\label{redu}
\eea
where
$$
O_{2}=\left(
        \begin{array}{cc}
          0 & 0 \\
          0 & 0 \\
        \end{array}
      \right),\ \
\mathcal{U}_{0}=\left(
                         \begin{array}{cc}
                           2 & 0 \\
                           0 & -1 \\
                         \end{array}
                       \right)
$$
corresponding the translation symmetry, dilation and rotation symmetry respectively. $\mathcal{U}$ is the essential part which reflects the stability of ERE, hence
the linear Hamiltonian system (\ref{EREH})
$$
\zeta'_{\mathfrak{e}}(\theta)=J_{4n}B(\theta)\zeta_{\mathfrak{e}}(\theta), \zeta_{\mathfrak{e}}(0)=I_{4n},
$$
can be decomposed to three subsystems with $B(\theta)=B_{1}(\theta)\diamond B_{2}(\theta)\diamond B_{3}(\theta)$, where
$$
B_{1}(\theta)=\left( \begin{array}{cccc} I_{2} & -\mathbb{J}_{1} \\
\mathbb{J}_{1} & I_{2}-\frac{I_{2}}{1+\mathfrak{e}\cos\theta}
\end{array}\right),\ \
B_{2}(\theta)=\left( \begin{array}{cccc} I_{2} & -\mathbb{J}_{1} \\
\mathbb{J}_{1} & I_{2}-\frac{I_{2}+\mathcal{U}_{0}
}{1+\mathfrak{e}\cos\theta}
\end{array}\right),\\
B_{3}(\theta)=\left( \begin{array}{cccc} I_{2n-4} & -\mathbb{J}_{n-2} \\
\mathbb{J}_{n-2} & I_{2n-4}-\frac{I_{2n-4}+\mathcal{U}
}{1+\mathfrak{e}\cos\theta}
\end{array}\right),
$$
%
%
%
Let $\eta_\mathfrak{e}(\theta)$ be the fundamental solution of $B_3$, that is
\bea \eta'_\mathfrak{e}(\theta)=JB_3(\theta)\eta_\mathfrak{e}(\theta),\ \ \eta_\mathfrak{e}(0)=I_{2n}. \label{ga3}\nonumber
\eea
The ERE is called spectral stable (linearly stable, hyperbolic, resp.), if the $\eta_\mathfrak{e}(2\pi)$ is spectral stable
(linearly stable, hyperbolic, resp.). In order to study the regular $n$-gon ERE, we should construct the remaining
column vectors in matrix $A$, we will give the construction in following section.

\section{Reduction of the regular $n$-gon ERE}\label{Se3}

Let $a=(x^T_{1},...x^T_{n})^T$ be the position vector of the $n$-gon central
configuration with $x_{k}=(\cos\theta_{k},\sin\theta_{k})^T$, where
$\theta_{k}=\frac{2\pi k}{n}, k\in\{1,2,...n\}$ and let $\mathcal{M}=I_{2n}$, i.e all the particle masses are assumed to be one.
Define
\bea
v(1)&=&(\cos2\theta_{1},\sin2\theta_{1},\cdots,\cos2\theta_{n},\sin2\theta_{n})^T,\nonumber \\
v(l)&=&(v_{1l}, \cdots , v_{nl})^T,w(l)=(w_{1l},\cdots,w_{nl})^T, \nonumber \\
v_{kl}&=&\cos \theta_{kl}\cdot(\cos \theta_{k},\sin \theta_{k}),\ \
w_{kl}=\sin \theta_{kl}\cdot(\cos \theta_{k},\sin \theta_{k}).\nonumber
\eea
%
Direct computations show that
\bea
v(1)^T\mathcal{M}v(1)=n,\ \
v(l)^T\mathcal{M}v(l)=\frac{n}{2},\ \ w(l)^T\mathcal{M}w(l)=\frac{n}{2}. \nonumber
\eea
Then we normalize this vectors as follows,
\bea
\frac{1}{\sqrt{n}}v(1),\ \
\sqrt{\frac{2}{n}}v(l),\ \ \sqrt{\frac{2}{n}}w(l).\nonumber
\eea
Now we construct the matrix $A$ as follow,
\bea\label{matrix A}
A=\left( \begin{array}{cccc}
A_{11} & A_{12} & \cdots & A_{1n} \\
A_{21} & A_{22} & \cdots & A_{2n}\\
\cdots & \cdots & \cdots & \cdots\\
A_{n1} & A_{n2} & \cdots & A_{nn}\end{array}\right),
\eea
where each $A_{ij}$ is defined by
\bea
A_{cen}=\left( \begin{array}{c}A_{11} \\ A_{21} \\\vdots \\A_{n1}\end{array}\right)
=\frac{1}{\sqrt{n}}(e_{1}, \mathbb{J}_{n}e_{1}), \ \
A(0)=\left( \begin{array}{c}A_{12} \\ A_{22} \\\vdots \\A_{n2}\end{array}\right)
=\frac{1}{\sqrt{n}}(a,\mathbb{J}_{n}a),\ \
A(1)=\left( \begin{array}{c}A_{13} \\ A_{23} \\\vdots \\A_{n3}\end{array}\right)
=\frac{1}{\sqrt{n}}(v(1),\mathbb{J}_{n}v(1)),\nonumber
\eea
\bea
A(l)=\left( \begin{array}{cc}A_{1(2l)} & A_{1(2l+1)} \\ A_{2(2l)} & A_{2(2l+1)} \\ \vdots & \vdots \\A_{n(2l)} &
A_{n(2l+1)}\end{array}\right)
&=&\sqrt{\frac{2}{n}}(v(l),\mathbb{J}_{n}v(l),w(l),\mathbb{J}_{n}w(l)),\ \
2\leq l\leq [\frac{n-1}{2}]\ \nonumber \\
A(\frac{n}{2})=\left( \begin{array}{c}A_{1n} \\ A_{2n} \\ \vdots  \\A_{nn} \end{array}\right)
&=&\sqrt{\frac{2}{n}}(v(\frac{n}{2}),\mathbb{J}_{n}v(\frac{n}{2})),\ \ \mathrm{if}\ \ n\in 2\mathbb{N}, \label{3.24}\nonumber
\eea
where $[x]=\max\{k\,|\, k\leq x, k\in\mathbb{N}\}$. Then the matrix $A$ satisfies $A^T\mathcal{M}A=I_{2n}$ and $A\mathbb{J}_{n}=\mathbb{J}_{n}A$ as required in (\ref{A}). Based on
this matrix $A$, we have following lemma
\begin{prop}\label{lem-redu}
By constructing matrix $A$ as above, we obtain
\bea
A^{-1}\frac{1}{\lambda}\mathcal{M}^{-1}D^2U(a)A=diag(O_{2},\mathcal{U}_{0},\mathcal{U}_{1},...,\mathcal{U}_{l},...,\mathcal{U}_{[\frac{n}{2}]}),\nonumber
\eea
where
\bea
\mathcal{U}_{0}=\left(\begin{array}{cc}2 & 0\\
0 & -1
\end{array}\right),\ \
\mathcal{U}_{1}=\frac{1}{\lambda}\left(\begin{array}{cc}z & 0\\
0 & z
\end{array}\right),\nonumber
\eea
\bea
\mathcal{U}_{l}=\frac{1}{\lambda}\left(\begin{array}{cccc}a_{l} & 0 & 0 & S_{l}\\
0 & b_{l} & -S_{l} & 0\\ 0 & -S_{l} & a_{l} & 0\\
S_{l} & 0 & 0 & b_{l}\end{array}\right),\ \ \begin{array}{c} 2\leq l\leq [\frac{n-1}{2}]
\end{array},\nonumber
\eea
\bea
\mathcal{U}_{[\frac{n}{2}]}=\frac{1}{\lambda}\left(\begin{array}{cc}
P_{\frac{n}{2}}-3Q_{\frac{n}{2}} & 0 \\ 0 & P_{\frac{n}{2}}+3Q_{\frac{n}{2}}\end{array}\right),\nonumber
 \ \      \mathrm{if}\ \ n\in 2\mathbb{N},
\eea
with
\bea
z=\sum_{j=1}^{n-1}\frac{1}{2d^{3}_{nj}}(1-\cos2\theta_{j}),\ \ a_{l}=P_{l}-3Q_{l},\ \ b_{l}=P_{l}+3Q_{l},\ \ \lambda=\frac{1}{4}\sum_{j=1}^{n-1}\csc \frac{\pi j}{n}.\nonumber
\eea
\bea
P_{l}=\sum_{j=1}^{n-1}\frac{1-\cos \theta_{jl}\cos \theta_{j}}{2d_{nj}^3},\ \
S_{l}=\sum_{j=1}^{n-1}\frac{\sin \theta_{jl}\sin \theta_{j}}{2d_{nj}^3},\ \
Q_{l}=\sum_{j=1}^{n-1}\frac{\cos \theta_{j}-\cos \theta_{jl}}{2d_{nj}^3}.\nonumber
\eea
\end{prop}

Now, let $D^2U(a)$ denoted by the following form,
\bea
D^2U(a)=\left( \begin{array}{cccc} U_{11} & U_{12} & ... & U_{1n} \\
 U_{21} & U_{22} & ... & U_{2n}\\
 ... & ... & ... & ...\\
 U_{n1} & U_{n2} & ... & U_{nn}\end{array}\right)_{2n\times2n},\nonumber
\eea
Then $U_{ij}=U_{ji},U_{jj}=-\sum_{i\neq j}U_{ij}$ and for $1\leq i\neq j\leq N$, we have $U_{ij}=\frac{m_{i}m_{j}}{d_{ij}^3}(I_{2}-3x_{ij}x_{ij}^T)$,
where $x_{ij}=\frac{x_{j}-x_{i}}{d_{ij}}.$ Recall that $R(\theta)=\left(
                                                                    \begin{array}{cc}
                                                                      \cos\theta & -\sin\theta \\
                                                                      \sin\theta & \cos\theta \\
                                                                    \end{array}
                                                                  \right)
$ and we introduce a new matrix
$
\hat{R}(\theta)=\left(
\begin{array}{cc}
\cos\theta & \sin\theta \\
\sin\theta & -\cos\theta \\
\end{array}
\right),
$
then direct computation shows that
\bea
U_{ij}
&=&\frac{1}{d_{ij}^3}\left(-\frac{1}{2}I_{2}+\frac{3}{2}R(\theta_{j-i})\hat{R}(2\theta_{i})\right),i\neq j,\label{3.29}\\
U_{jj}&=&-\sum_{i\neq j}U_{ij}, \label{3.30}
\eea
To obtain Proposition \ref{lem-redu}, we need the following lemmas.
\begin{lem}\label{lem-l=1}
For $l=1$, we have
\bea
A(1)^{-1}\frac{1}{\lambda}\mathcal{M}^{-1}D^2U(a)A(1)=\frac{1}{\lambda}\left(
       \begin{array}{cc}
         z & 0 \\
         0 & z \\
       \end{array}
     \right).\label{l=1}\nonumber
\eea
where
$
z=\sum_{j=1}^{n-1}\frac{1}{2d^{3}_{nj}}(1-\cos2\theta_{j}).
$
\end{lem}
\begin{proof}
For $l = 1$, we have
$$
\frac{1}{\lambda} \mathcal{M}^{-1} D^{2} U (a) A (1) =\frac{1}{\sqrt{n} \lambda}\left (\begin{array}{c}
\sum_{j=1}^{n} U_{1 j} R\left (2 \theta_{j}\right) \\
\vdots \\
\sum_{j=1}^{n} U_{n j} R\left (2 \theta_{j}\right)
\end{array}\right).
$$
from (\ref{3.29}), (\ref{3.30}), direct computation shows that,
\bea
\begin{aligned}
        \sum_{j=1}^{n}U_{ij}R (2\theta_j)
        &=\sum_{j=1, j\ne i}^{n}\frac{1}{d_{ij}^3} \left(-\frac{1}{2}I_2+\frac{3}{2}R (\theta_{j-i}) \hat{R} (2\theta_i) \right) \left(R (2\theta_j) -R (2\theta_i) \right) \\
        &=\sum_{j=1, j\ne i}^{n}\frac{1}{2d_{ij}^3} \left(R (2\theta_i) -R (2\theta_j) \right) +\frac{3}{2}\sum_{j=1, j\ne i}^{n}\frac{1}{d_{ij}^3} \left(\hat{R} (\theta_{i-j}) -\hat{R} (\theta_{j-i}) \right) \\
        &=\begin{aligned}[t]
            \sum_{j=1, j\ne i}^{n}\frac{1}{2d_{ij}^3} \left(I_2-R (2 \theta_{j-i} ) \right) R (2\theta_i)
            +\frac{3}{2}\sum_{j=1, j\ne i}^{n}\frac{1}{d_{ij}^3}\begin{pmatrix}
  0& 2\sin  \theta_{i-j} \\
  2\sin  \theta_{i-j} &0
\end{pmatrix},
        \end{aligned}
\end{aligned}\nonumber
\eea
the second equality from the facts
$$
R (\theta_{j-i}) \hat{R} (2\theta_i)R (2\theta_j)=\hat{R} (\theta_{i-j}),\ \ R (\theta_{j-i}) \hat{R} (2\theta_i)R (2\theta_i)=\hat{R} (\theta_{j-i}).
$$
Moreover, one can check that $ \sum_{j=1, j\ne i}^{n}\frac{1}{d_{ij}^3}\sin \theta_{i-j}=0$, hence
\bea
\sum_{j=1}^{n}U_{ij}R (2\theta_j)= \sum_{j=1, j\ne i}^{n}\frac{1}{2d_{ij}^3} \left(I_2-R (2 \theta_{j-i} ) \right) R (2\theta_i).\nonumber
\eea
Since $d_{ij}=2|\sin\frac{\theta_{i}-\theta_{j}}{2}|=d_{n(j+n-i)}$, we further obtain
\bea
 \begin{aligned}
        \sum_{j=1}^{n}U_{ij}R (2\theta_j)&=\sum_{j=1, j\ne i}^{n}\frac{1}{2d_{n (j+n-i) }^3} \left(I_2-R (2 \theta_{j-i} ) \right) R (2\theta_i)
        =\sum_{\tilde{j}=1+n-i, \tilde{j}\ne n}^{n+n-i}\frac{1}{2d_{n\tilde{j}}^3} \left(I_2-R (2 \theta_{\tilde{j}} ) \right) R (2\theta_i) \\
        &=\sum_{\tilde{j}=1+n-i}^{n-1}\frac{1}{2d_{n\tilde{j}}^3} \left(I_2-R (2 \theta_{\tilde{j}} ) \right) R (2\theta_i)
            +\sum_{\tilde{j}=n+1}^{n+n-i}\frac{1}{2d_{n\tilde{j}}^3} \left(I_2-R (2 \theta_{\tilde{j}} ) \right) R (2\theta_i)\\
            &=\sum_{\tilde{j}=1+n-i}^{n-1}\frac{1}{2d_{n\tilde{j}}^3} \left(I_2-R (2 \theta_{\tilde{j}} ) \right) R (2\theta_i)
            +\sum_{\hat{j}=1}^{n-i}\frac{1}{2d_{n (\hat{j}+n) }^3} \left(I_2-R (2 \theta_{ (\hat{j}+n) } ) \right) R (2\theta_i)\\
            &=\sum_{\tilde{j}=1+n-i}^{n-1}\frac{1}{2d_{n\tilde{j}}^3} \left(I_2-R (2 \theta_{\tilde{j}} ) \right) R (2\theta_i)
            +\sum_{\hat{j}=1}^{n-i}\frac{1}{2d_{n \hat{j} }^3} \left(I_2-R (2 \theta_{ \hat{j} } ) \right) R (2\theta_i)\\
            &=\sum_{j=1}^{n-1}\frac{1}{2d_{nj}^3}\begin{pmatrix}
            1-\cos2\theta_j&\sin2\theta_j\\
            -\sin2\theta_j&1-\cos2\theta_j
             \end{pmatrix}R (2\theta_i).
    \end{aligned}\label{simcal}
\eea
Since $\sum_{j=1}^{n-1}\frac{1}{2d_{nj}^3}\sin2\theta_j=0$, we further obtain
$$
\sum_{j=1}^{n}U_{ij}R (2\theta_j)
        =\sum_{j=1}^{n-1}\frac{1}{2d_{nj}^3}\begin{pmatrix}
            1-\cos2\theta_j&0\\
            0&1-\cos2\theta_j
        \end{pmatrix}R (2\theta_i)
        =\begin{pmatrix}
            z&0\\0&z
        \end{pmatrix}R (2\theta_i),
$$
where
$
z=\sum_{j=1}^{n-1}\frac{1}{2d^{3}_{nj}}(1-\cos2\theta_{j}).
$
Therefore, we have
$$
\frac{1}{\lambda} \mathcal{M}^{-1} D^{2} U (a) A (1) =\frac{1}{\sqrt{n} \lambda}\left (\begin{array}{c}
\sum_{j=1}^{n} U_{1 j} R\left (2 \theta_{j}\right) \\
\vdots \\
\sum_{j=1}^{n} U_{n j} R\left (2 \theta_{j}\right)
\end{array}\right) =\frac{A (1) }{\lambda}\left (\begin{array}{cc}
z & 0 \\
0 & z
\end{array}\right),
$$
this implies
$$
A(1)^{-1}\frac{1}{\lambda}\mathcal{M}^{-1}D^2U(a)A(1)=\frac{1}{\lambda}\left(
       \begin{array}{cc}
         z & 0 \\
         0 & z \\
       \end{array}
     \right).\label{l=1}\nonumber
$$
\end{proof}
\begin{lem}\label{lem-1-n}
For $2\leq l\leq[\frac{n-1}{2}]$, we have
$$
A(l)^{-1}\frac{1}{\lambda}\mathcal{M}^{-1}D^2U(a)A(l)=
\frac{1}{\lambda}\left(\begin{array}{cccc}a_{l} & 0 & 0 & S_{l}\\
0 & b_{l} & -S_{l} & 0\\ 0 & -S_{l} & a_{l} & 0\\
S_{l} & 0 & 0 & b_{l}\end{array}\right)\label{l=2}
$$
where
\bea
a_{l}=P_{l}-3Q_{l},\ \ b_{l}=P_{l}+3Q_{l},\nonumber
\eea
\bea
P_{l}=\sum_{j=1}^{n-1}\frac{1-\cos \theta_{jl}\cos \theta_{j}}{2d_{nj}^3},\ \
S_{l}=\sum_{j=1}^{n-1}\frac{\sin \theta_{jl}\sin \theta_{j}}{2d_{nj}^3},\ \
Q_{l}=\sum_{j=1}^{n-1}\frac{\cos \theta_{j}-\cos \theta_{jl}}{2d_{nj}^3}.\label{3.5}\nonumber
\eea
\end{lem}
\begin{proof}
For $2 \leq l \leq\left[\frac{n - 1}{2}\right]$, consider
$$\mathcal{M}^{-1}D^2U (a) \begin{pmatrix}
    \omega^l\hat{R} (\theta_1) \\ \vdots \\\omega^{nl}\hat{R} (\theta_n)
\end{pmatrix}=\begin{pmatrix}
    \sum_{j=1}^{n}U_{1j}\omega^{jl}\hat{R} (\theta_j) \\ \vdots \\ \sum_{j=1}^{n}U_{nj}\omega^{jl}\hat{R} (\theta_j)
\end{pmatrix}, $$
where $\omega=e^{\frac{2\pi}{n}\sqrt{-1}}$, $\sqrt{-1}$ represents the imaginary unit. From (\ref{3.29}), (\ref{3.30}), direct computation shows that,
\bea
\begin{aligned}
    \sum_{j=1}^{n}U_{ij}\omega^{jl}\hat{R} (\theta_j)
    &=\sum_{j=1, j\ne i}^{n} \left(U_{ij}\omega^{jl}\hat{R} (\theta_j) -U_{ij}\omega^{il}\hat{R} (\theta_i) \right)
    =\sum_{j=1, j\ne i}^{n}U_{ij} \left(\omega^{ (j-i) l}R (\theta_{j-i}) -I\right) \omega^{il}\hat{R} (\theta_i) \\
    &=\sum_{j=1, j\ne i}^{n}\frac{1}{d_{ij}^3} \left(-\frac{1}{2}I+\frac{3}{2}R (\theta_{j-i}) \hat{R} (2\theta_i) \right) \cdot  \left(\omega^{ (j-i) l}R (\theta_{j-i}) -I\right) \omega^{il}\hat{R} (\theta_i) \\
    &=\begin{aligned}[t]
        &-\sum_{j=1, j\ne i}^{n}\frac{1}{2d_{ij}^3}  \left(\omega^{ (j-i) l}R (\theta_{j-i}) -I\right) \omega^{il}\hat{R} (\theta_i) \\
        & +\frac{3}{2}\sum_{j=1, j\ne i}^{n}\frac{R (\theta_{j-i}) \hat{R} (2\theta_i) }{d_{ij}^3}  \left(\omega^{ (j-i) l}R (\theta_{j-i}) -I\right) \omega^{il}\hat{R} (\theta_i)
    \end{aligned}\\
    &=\begin{aligned}[t]
        &-\sum_{j=1, j\ne i}^{n}\frac{1}{2d_{ij}^3}  \left(\omega^{ (j-i) l}R (\theta_{j-i}) -I\right) \omega^{il}\hat{R} (\theta_i) \\
        & +\frac{3}{2}\sum_{j=1, j\ne i}^{n}\frac{1}{d_{ij}^3}  \left(\omega^{ (j-i) l}I-R (\theta_{j-i}) \right) \omega^{il}\hat{R} (2\theta_i) \hat{R} (\theta_i),
    \end{aligned}\label{simcal2}
\end{aligned}
\eea
the last equation is based on the fact
$$
R (\theta_{j-i}) \hat{R} (2\theta_i) R (\theta_{j-i})=\hat{R} (2\theta_i).
$$
Similar the calculations of (\ref{simcal}) in the proof of Lemma \ref{lem-l=1}, we have
\bea
\begin{aligned}
    &\ \ \ -\sum_{j=1, j\ne i}^{n}\frac{1}{2d_{ij}^3}  \left(\omega^{ (j-i) l}R (\theta_{j-i}) -I\right) \omega^{il}\hat{R} (\theta_i)\\
    &=-\sum_{j=1}^{n-1}\frac{1}{2d_{nj}^3}  (\omega^{ (j-n) l}R (\theta_{j-n}) -I) \omega^{il}\hat{R} (\theta_i)
    =-\sum_{j=1}^{n-1}\frac{1}{2d_{nj}^3}  (\omega^{jl}R (\theta_j) -I) \omega^{il}\hat{R} (\theta_i) \\
    &=
    \begin{aligned}[t]
        &-\sum_{j=1}^{n-1}\frac{1}{2d_{nj}^3}\Bigg(\begin{pmatrix}
        \cos\theta_{jl}\cos\theta_j&-\cos\theta_{jl}\sin\theta_j\\
        \cos\theta_{jl}\sin\theta_{j}&\cos\theta_{jl}\cos\theta_j
    \end{pmatrix}
    +\sqrt{-1}\begin{pmatrix}
        \sin\theta_{jl}\cos\theta_j&-\sin\theta_{jl}\sin\theta_j\\
        \sin\theta_{jl}\sin\theta_{j}&\sin\theta_{jl}\cos\theta_j
    \end{pmatrix}-I \Bigg)\omega^{il}\hat{R} (\theta_i).
    \end{aligned}\\
&=
    -\sum_{j=1}^{n-1}\frac{1}{2d_{nj}^{3}}\begin{pmatrix}
        \cos\theta_{jl}\cos\theta_j-1 & -\sqrt{-1}\sin\theta_{jl}\sin\theta_j\\
        \sqrt{-1}\sin\theta_{jl}\sin\theta_j & \cos\theta_{jl}\cos\theta_j-1
    \end{pmatrix}\omega^{il}\hat{R} (\theta_i)
    =\begin{pmatrix}
        P_l&\sqrt{-1}S_l\\-\sqrt{-1}S_l&P_l
    \end{pmatrix}\omega_{il}\hat{R} (\theta_i), \label{simcal3}
\end{aligned}
\eea
where
$$P_l=\sum_{j=1}^{n-1}\frac{1-\cos\theta_{jl}\cos\theta_j}{2d_{nj}^3}, \quad S_l=\sum_{j=1}^{n-1}\frac{\sin\theta_{jl}\sin\theta_j}{2d_{nj}^3}.$$
The third equation is based on the fact
$$\sum_{j=1}^{n-1}\frac{1}{2d_{nj}^3}\cos\theta_{jl}\sin\theta_j=0,\ \ \sum_{j=1}^{n-1}\frac{1}{2d_{nj}^3}\sin\theta_{jl}\cos\theta_j=0.$$
Similar, we have
\bea
\begin{aligned}
    &\ \ \ \frac{3}{2}\sum_{j=1, j\ne i}^{n}\frac{1}{d_{ij}^3}\left (\omega^{ (j-i) l}I-R (\theta_{j-i}) \right) \omega^{il}\hat{R} (2\theta_i) \hat{R} (\theta_i)\\
    &=\frac{3}{2}\sum_{j=1}^{n-1}\frac{1}{d_{nj}^3}\left (\omega^{ (j-n) l}I-R (\theta_{j-n}) \right) \omega^{il}\hat{R} (2\theta_i) \hat{R} (\theta_i)
    =\frac{3}{2}\sum_{j=1}^{n-1}\frac{1}{d_{nj}^3}\left (\omega^{jl}I-R (\theta_{j}) \right) \omega^{il}\hat{R} (2\theta_i) \hat{R} (\theta_i)\\
    &=\frac{3}{2}\sum_{j=1}^{n-1}\frac{1}{d_{nj}^3}\begin{pmatrix}
        \cos\theta_{jl}-\cos\theta_j & 0\\ 0 & \cos\theta_{jl}-\cos\theta_j
    \end{pmatrix}\omega^{il}R (\theta_i)
    =\begin{pmatrix}
        -3Q_l&0\\0&-3Q_l
    \end{pmatrix}\omega^{il}R (\theta_i),\label{simcal4}
\end{aligned}
\eea
where
$$Q_l=\sum_{j=1}^{n-1}\frac{\cos\theta_j-\cos\theta_{jl}}{2d_{nj}^3}.$$
The third equation is based on the fact
$$\sum_{j=1}^{n-1}\frac{\sin\theta_{jl}}{d_{nj}^3}=0,\ \ \sum_{j=1}^{n-1}\frac{\sin\theta_{j}}{d_{nj}^3}=0.$$
Substituting equations (\ref{simcal3}) and (\ref{simcal4}) into equation (\ref{simcal2}) yields
\bea
\begin{aligned}
\sum_{j=1}^{n}U_{ij}\omega^{jl}\hat{R} (\theta_j)&=\begin{pmatrix}
    P_l&\sqrt{-1}S_l\\-\sqrt{-1}S_l&P_l
\end{pmatrix}\omega^{il}\hat{R} (\theta_i) +\begin{pmatrix}
    -3Q_l&0\\0&-3Q_l
\end{pmatrix}\omega^{il}R (\theta_i)\\
&=\left (P_lI_2-\sqrt{-1}J_2S_l\right) \omega^{il}\hat{R} (\theta_i)
-3Q_l \omega^{il}R (\theta_i).
\end{aligned}
\label{simcal5}
\eea
One can check that
$$
\begin{aligned}
    \omega^{il}\hat{R} (\theta_i)
    &=\cos\theta_{il}\begin{pmatrix}
        \cos\theta_i&\sin\theta_i\\
        \sin\theta_i&-\cos\theta_i
    \end{pmatrix}+\sqrt{-1}\sin\theta_{il}\begin{pmatrix}
        \cos\theta_i&\sin\theta_i\\
        \sin\theta_i&-\cos\theta_i
    \end{pmatrix}\\
    &=\left(v_{il}+\sqrt{-1}w_{il}, -J_2 (v_{il}+\sqrt{-1}w_{il}) \right),
\end{aligned}
$$
$$
\begin{aligned}
    \omega^{il}R (\theta_{i})
    &=\cos\theta_{il}\begin{pmatrix}
        \cos\theta_i&-\sin\theta_i\\
        \sin\theta_i&\cos\theta_i
    \end{pmatrix}+\sqrt{-1}\sin\theta_{il}\begin{pmatrix}
        \cos\theta_i&-\sin\theta_i\\
        \sin\theta_i&\cos\theta_i
    \end{pmatrix}\\
    &=\left(v_{il}+\sqrt{-1}w_{il}, J_2 (v_{il}+\sqrt{-1}w_{il}) \right),
\end{aligned}
$$
hence (\ref{simcal5}) is equivalent to
$$
\begin{aligned}
&\ \ \ \sum_{j=1}^{n}U_{ij}\left(v_{il}+\sqrt{-1}w_{il}, -J_2 (v_{il}+\sqrt{-1}w_{il})\right)\\
&=\left (P_lI_2-\sqrt{-1}J_2S_l\right) \cdot\left (v_{il}+\sqrt{-1}w_{il}, -J_2 (v_{il}+\sqrt{-1}w_{il}) \right)
-3Q_l\cdot\left (v_{il}+\sqrt{-1}w_{il}, J_2 (v_{il}+\sqrt{-1}w_{il}) \right)\\
&=\left (v_{il}+\sqrt{-1}w_{il}, J_2 (v_{il}+\sqrt{-1}w_{il}) \right)
    \cdot\begin{pmatrix}
        P_l-3Q_l&-\sqrt{-1}S_l\\-\sqrt{-1}S_l&-P_l-3Q_l
    \end{pmatrix}.
\end{aligned}
$$
This implies
$$
\begin{aligned}
    &\sum_{j=1}^{n} U_{i j}\left (v_{j l},  J_{2} v_{j l},  w_{j l},  J_{2} w_{j l}\right)
    =\left (v_{i l},  J_{2} v_{i l},  w_{i l},  J_{2} w_{i l}\right)\begin{pmatrix}
        P_l-3Q_l&0&0&S_l\\
        0&P_l+3Q_l&-S_l&0\\
        0&-S_l&P_l-3Q_l&0\\
        S_l&0&0&P_l+3Q_l
    \end{pmatrix},
\end{aligned}
$$
Therefore we have
\bea
\begin{aligned}
    &\mathcal{M}^{-1}D^2U (a) \left (v (l),  \mathbb{J}_nv (l),  w (l),  \mathbb{J}_nw (l) \right) \\
    =&\left (v (l),  \mathbb{J}_nv (l),  w (l),  \mathbb{J}_nw (l) \right) \begin{pmatrix}
        P_l-3Q_l&0&0&S_l\\
        0&P_l+3Q_l&-S_l&0\\
        0&-S_l&P_l-3Q_l&0\\
        S_l&0&0&P_l+3Q_l
    \end{pmatrix}.
\end{aligned}\label{for-n}
\eea
and thus we have
$$
A(l)^{-1}\frac{1}{\lambda}\mathcal{M}^{-1}D^2U(a)A(l)=
\frac{1}{\lambda}\left(\begin{array}{cccc}a_{l} & 0 & 0 & S_{l}\\
0 & b_{l} & -S_{l} & 0\\ 0 & -S_{l} & a_{l} & 0\\
S_{l} & 0 & 0 & b_{l}\end{array}\right)
$$
where
\bea
a_{l}=P_{l}-3Q_{l},\ \ b_{l}=P_{l}+3Q_{l},\nonumber
\eea
$$
P_{l}=\sum_{j=1}^{n-1}\frac{1-\cos \theta_{jl}\cos \theta_{j}}{2d_{nj}^3},\ \
S_{l}=\sum_{j=1}^{n-1}\frac{\sin \theta_{jl}\sin \theta_{j}}{2d_{nj}^3},\ \
Q_{l}=\sum_{j=1}^{n-1}\frac{\cos \theta_{j}-\cos \theta_{jl}}{2d_{nj}^3}.
$$
This completes the proof.
\end{proof}
\begin{lem}\label{lem-1-n/2}
For $n \in 2 \mathbb{N}$, we have
\begin{equation*}
A (\frac{n}{2}) ^{-1} \frac{1}{\lambda} \mathcal{M}^{-1} D^{2} U (a) A (\frac{n}{2}) =\frac{1}{\lambda}\left (\begin{array}{cc}
a_{\frac{n}{2}} & 0 \\
0 & b_{\frac{n}{2}}
\end{array}\right),
\end{equation*}
where
\bea
a_{\frac{n}{2}}=P_{\frac{n}{2}}-3Q_{\frac{n}{2}},\ \ b_{\frac{n}{2}}=P_{\frac{n}{2}}+3Q_{\frac{n}{2}},\nonumber
\eea
$$
P_{\frac{n}{2}}=\sum_{j=1}^{n-1}\frac{1-\cos \theta_{j\frac{n}{2}}\cos \theta_{j}}{2d_{nj}^3},\ \
Q_{\frac{n}{2}}=\sum_{j=1}^{n-1}\frac{\cos \theta_{j}-\cos \theta_{j\frac{n}{2}}}{2d_{nj}^3}.
$$
\end{lem}
\begin{proof}
Following the proof of Lemma \ref{lem-1-n}, we obtain an similar formula analogous to (\ref{for-n}) as follows,
$$
\begin{aligned}
    &\mathcal{M}^{-1}D^2U (a) \left (v (\frac{n}{2}),  \mathbb{J}_nv (\frac{n}{2}),  w (\frac{n}{2}),  \mathbb{J}_nw (\frac{n}{2}) \right) \\
    =&\left (v (\frac{n}{2}),  \mathbb{J}_nv (\frac{n}{2}),  w (\frac{n}{2}),  \mathbb{J}_nw (\frac{n}{2}) \right) \begin{pmatrix}
        P_{\frac{n}{2}}-3Q_{\frac{n}{2}}&0&0&S_{\frac{n}{2}}\\
        0&P_{\frac{n}{2}}+3Q_{\frac{n}{2}}&-S_{\frac{n}{2}}&0\\
        0&-S_{\frac{n}{2}}&P_{\frac{n}{2}}-3Q_{\frac{n}{2}}&0\\
        S_{\frac{n}{2}}&0&0&P_{\frac{n}{2}}+3Q_{\frac{n}{2}}
    \end{pmatrix}.
\end{aligned}
$$
From the definition of $w(l)$, one can see that $w(\frac{n}{2})=0$, hence we obtain
$$
\begin{aligned}
    \mathcal{M}^{-1}D^2U (a) \left (v (\frac{n}{2}),  \mathbb{J}_nv (\frac{n}{2}) \right)
    =\left (v (\frac{n}{2}),  \mathbb{J}_nv (\frac{n}{2}) \right) \begin{pmatrix}
        P_{\frac{n}{2}}-3Q_{\frac{n}{2}}&0\\
        0&P_{\frac{n}{2}}+3Q_{\frac{n}{2}}
    \end{pmatrix}.
\end{aligned}
$$
and thus we have
\begin{equation*}
A ({\frac{n}{2}}) ^{-1} \frac{1}{\lambda} \mathcal{M}^{-1} D^{2} U (a) A ({\frac{n}{2}}) =\frac{1}{\lambda}\left (\begin{array}{cc}
a_{{\frac{n}{2}}} & 0 \\
0 & b_{{\frac{n}{2}}}
\end{array}\right).
\end{equation*}
\end{proof}
Based on the above Lemmas, we can readily establish the proof of Proposition \ref{lem-redu} as follows,
\begin{proof}
As we see in (\ref{redu}), we have
$$
A_{cen}^{-1}\frac{1}{\lambda}\mathcal{M}^{-1}D^2U(a)A_{cen}=O_{2},\ \ A(0)^{-1}\frac{1}{\lambda}\mathcal{M}^{-1}D^2U(a)A(0)=\mathcal{U}_{0}=\left(
                         \begin{array}{cc}
                           2 & 0 \\
                           0 & -1 \\
                         \end{array}
                       \right).
$$
For $1\leq l\leq [\frac{n}{2}]$, based on lemma \ref{lem-l=1}, Lemma \ref{lem-1-n} and Lemma \ref{lem-1-n/2}, we can compute $A(l)^{-1}\frac{1}{\lambda}\mathcal{M}^{-1}D^2U(a)A(l)$,
hence we obtain
$$
A^{-1}\frac{1}{\lambda}\mathcal{M}^{-1}D^2U(a)A=diag(O_{2},\mathcal{U}_{0},\mathcal{U}_{1},...,\mathcal{U}_{l},...,\mathcal{U}_{[\frac{n}{2}]}),\nonumber
$$
this completes the proof of Proposition \ref{lem-redu}.
\end{proof}
Now, we give the proof of Theorem \ref{redu-n-gon}.
\begin{proof}
From Theorem \ref{red-ERE}, the linear Hamiltonian system of $ERE$ have following form
\bea
\zeta'_{\mathfrak{e}}(\theta)=J_{4n}B(\theta)\zeta_{\mathfrak{e}}(\theta), \zeta_{\mathfrak{e}}(0)=I_{4n}.\nonumber
\eea
where
\bea B(\theta)=\left( \begin{array}{cccc} I_{2n} & -\mathbb{J}_{n} \\
\mathbb{J}_{n} & I_{2n}-\frac{I_{2n}+\frac{1}{\lambda}A^{T}D^2U(a)A}{1+\mathfrak{e}\cos\theta}
\end{array}\right),\quad \theta\in[0,2\pi],\nonumber
\eea
For the regular $n$-gon ERE, $A$ is constructed by (\ref{matrix A}), then from Proposition \ref{lem-redu}, we have
$$B(\theta)=B_{1}(\theta)\diamond B_{2}(\theta)\diamond B_{3}(\theta),$$ where
$$
B_{1}(\theta)=\left( \begin{array}{cccc} I_{2} & -\mathbb{J}_{1} \\
\mathbb{J}_{1} & I_{2}-\frac{I_{2}}{1+\mathfrak{e}\cos\theta}
\end{array}\right),\ \
B_{2}(\theta)=\left( \begin{array}{cccc} I_{2} & -\mathbb{J}_{1} \\
\mathbb{J}_{1} & I_{2}-\frac{I_{2}+\mathcal{U}_{0}
}{1+\mathfrak{e}\cos\theta}
\end{array}\right),\\
$$
\bea
B_{3}(\theta)=\mathcal{B}_1(\theta)\diamond\cdots\diamond \mathcal{B}_{[\frac{n}{2}]}(\theta), \nonumber\eea
with
\bea
\mathcal{B}_{l}(\theta)=\left( \begin{array}{cccc} I & -\mathbb{J}\\
\mathbb{J} & I-\frac{I+\mathcal{U}_{l}
}{1+\mathfrak{e}\cos\theta}  \end{array}\right), \, l=1,\cdots,[\frac{n}{2}].\nonumber  \eea
where $\mathcal{U}_{l}$ is given by Proposition \ref{lem-redu}. Here we omit the sub-indices of $I$ and $\mathbb{J}$, which are chosen to have the same dimensions as those of $\mathcal{U}_l$. This completes the proof of Theorem \ref{redu-n-gon}.
\end{proof}

\section{Spectral instability of the regular $n$-gon ERE}
Let consider the following operator
\bea \label{1-dim-oper}
\mathcal{A}(\delta,\mathfrak{e})=-\frac{d^{2}}{d\theta^{2}}-1+\frac{\delta}{1+\mathfrak{e}\cos \theta},\ \
\delta>1, \ \ \mathfrak{e}\in[0,1),\eea
which is a self-adjoint operator with domain
$$\bar{D}_{1}(\omega, 2\pi)=\{y\in W^{2,2}([0,2\pi],\mathbb{C})|
y(2\pi)=\omega y(0),\dot{y}(2\pi)=\omega\dot{y}(0)\},$$
where $\omega\in\mathbb{U}$ and $W^{2,2}([0,2\pi])$ is the usual Sobolev space. Also, we consider the fundamental matrix $\gamma_{\delta,\mathfrak{e}}(\theta)$ of the first order linear Hamiltonian system corresponding to $\mathcal{A}(\delta,\mathfrak{e})$,
it satisfies
\bea \label{subha}
\gamma'_{\delta,\mathfrak{e}}(\theta)=J_{2}
\left(
\begin{array}{cc}
    1 & 0 \\
    0 & 1-\frac{\delta}{1+\mathfrak{e}\cos \theta}\\
  \end{array}\right)\gamma_{\delta,\mathfrak{e}}(\theta),\ \
\gamma_{\delta,\mathfrak{e}}(0)=I_{2}.
\eea
In order to proof Theorem \ref{unstable}, we need the following useful lemma. This lemma is first proved by the first author in \cite{O} in the study the hyperbolicity of the
Lagrange solution by the Maslov-type index theory.
\begin{lem}(\cite{O})\label{lem-hyp}
For any $\delta>1, \mathfrak{e}\in[0, 1)$ and $\omega\in \mathbb{C}$, the operator
$\mathcal{A}(\delta,\mathfrak{e})$ is a strictly positive operator on its domain $\bar{D}_{1}(\omega,2\pi)$. The monodromy matrix $\gamma_{\delta, \mathfrak{e}}(2\pi)$ of
the linear Hamiltonian system (\ref{subha}) is hyperbolic, more precisely, the eigenvalues of $\gamma_{\delta, \mathfrak{e}}(2\pi)$ are $\lambda, \lambda^{-1}$, where $\lambda>0$.
\end{lem}
Now, we give the proof of Theorem \ref{unstable}.
\begin{proof}
From Theorem \ref{redu-n-gon}, we know the fundamental solution of the essential part of the regular $n$-gon system in Theorem \ref{red-ERE}
satisfies
\bea \eta'_e(\theta)=JB_3(\theta)\eta_e(\theta),\ \ \eta_e(0)=I_{2n}. \nonumber
\eea
and we have the decomposition,
$$
\eta_{e}(\theta)=\eta_{1,\mathfrak{e}}(\theta)\diamond\eta_{2,\mathfrak{e}}(\theta)\diamond\cdots\eta_{[\frac{n}{2}],\mathfrak{e}}(\theta)
$$
where
$$
\eta'_{l,\mathfrak{e}}(\theta)=J\mathcal{B}_l(\theta)\eta_{1,e}(\theta),\ \ \eta_{l,\mathfrak{e}}(0)=I_{2}.
$$
\bea
\mathcal{B}_{l}(\theta)=\left( \begin{array}{cccc} I & -\mathbb{J}\\
\mathbb{J} & I-\frac{I+\mathcal{U}_{l}
}{1+\mathfrak{e}\cos\theta}  \end{array}\right), \, l=1,\cdots,[\frac{n}{2}].\nonumber  \eea
where $\mathcal{U}_{l}$ is given by Proposition \ref{lem-redu}. Here we omit the sub-indices of $I$ and $\mathbb{J}$, which are chosen to have the same dimensions as those of $\mathcal{U}_l$.

Consider $l=1$,
$$
\eta'_{1,\mathfrak{e}}(\theta)=J\mathcal{B}_1(\theta)\eta_{1,e}(\theta),\ \ \eta_{1,\mathfrak{e}}(0)=I_{2}.
$$
with
\bea
\mathcal{B}_{1}(\theta)=\left( \begin{array}{cccc} I_{2} & -\mathbb{J}_{1}\\
\mathbb{J}_{1} & I_{2}-\frac{I_{2}+\mathcal{U}_{1}
}{1+e\cos\theta}  \end{array}\right),\ \
\mathcal{U}_{1}=\frac{1}{\lambda}\left(\begin{array}{cc}z & 0\\
0 & z
\end{array}\right),\nonumber
\eea
and $z=\sum_{j=1}^{n-1}\frac{1}{2d^{3}_{nj}}(1-\cos2\theta_{j}), \lambda=\frac{1}{4}\sum_{j=1}^{n-1}\csc \frac{\pi j}{n}$. Let $\hat{\eta}_{1,\mathfrak{e}}=\left(
                            \begin{array}{cc}
                              \mathfrak{R}(\theta) & O \\
                              O & \mathfrak{R}(\theta)\\
                            \end{array}
                          \right)\eta_{1,\mathfrak{e}}(\theta)$, then
\bea
\hat{\eta}'_{\mathfrak{e}}(\theta)=J_{4}\left( \begin{array}{cccc} I_{2} & O_{2} \\
O_{2} & I_{2}-\frac{I_{2}+\mathcal{U}_{1}}{1+\mathfrak{e}\cos\theta}
\end{array}\right)\hat{\eta}_{1,\mathfrak{e}}(\theta),\ \ \hat{\eta}_{1,\mathfrak{e}}(0)=I_{4}.\nonumber
\eea
Then
\bea
\hat{\eta}_{\mathfrak{e}}(\theta)=\gamma_{\delta,\mathfrak{e}}(\theta)\diamond\gamma_{\delta,\mathfrak{e}}(\theta)\nonumber
\eea
where $\gamma_{\delta,\mathfrak{e}}(\theta)$ satisfies (\ref{subha}) with $\delta=1+\frac{z}{\lambda}$, since $\frac{z}{\lambda}>0$, from Lemma \ref{lem-hyp}, we obtain
$\hat{\eta}_{1,\mathfrak{e}}(2\pi)$ is hyperbolic and so is $\eta_{1,\mathfrak{e}}(2\pi)$, hence $\eta_{\mathfrak{e}}(2\pi)$ is spectral unstable and
for $n\geq3$ and any $\mathfrak{e}\in[0,1)$, the regular $n$-gon system is spectral unstable. This complete proof of Theorem \ref{unstable}.
\end{proof}
\section{Hyperbolicity Estimation}
As mentioned in the previous section, Moeckel \cite{Moe1} shows that the regular $n$-gon ERE is spectral unstable for $n\geq3$ and $\mathfrak{e}=0$, hence our Theorem \ref{unstable} generalize
this unstable result for any $\mathfrak{e}\in[0,1)$. Moreover, \cite{Moe1} also obtain the hyperbolicity of the regular $3, 4$ and $5$-gon ERE for $\mathfrak{e}=0$.
It's natural to guess the hyperbolicity also holds for any $\mathfrak{e}\in[0,1)$.  Although, in [4] and [6], the authors has already established the hyperbolicity of the regular
$3$-gon and $4$-gon for any $e\in[0,1)$, respectively, the hyperbolicity of the regular $5$-gon is still remains open, it presents greater complexity than the $n = 3,4$ cases. In this section, we will give a unified proof of the hyperbolicity of the regular $3, 4$ and $5$-gon ERE for any eccentricity. To achieve this goal, we introduce the following $\beta$-system (\ref{beta-system}) and we will see that it is related to the system of the Lagrange solution.
\subsection{Hyperbolicity of the $\beta$-system}\label{subsec beta}

\bea
\gamma_{\beta,\mathfrak{e}}^{\prime} (\theta) =J \mathcal{B}_{\beta,\mathfrak{e}} (\theta) \gamma_{\beta,\mathfrak{e}},\ \ \gamma_{\beta,\mathfrak{e}} (0) =I_{4} .\label{beta-system}
\eea
where
$$
\mathcal{B}_{\beta,\mathfrak{e}} (\theta) =\left (\begin{array}{cc}
I_{2} & -J_{2} \\
J_{2} & I_{2}-\frac{\mathcal{R}_{\beta}}{1+\mathfrak{e} \cos \theta}
\end{array}\right),\ \ \mathcal{R}_{\beta}=\frac{3}{2}I_{2}+\beta\left(
                                                                   \begin{array}{cc}
                                                                     -1 & 0 \\
                                                                     0 & 1 \\
                                                                   \end{array}
                                                                 \right),\ \ \beta\geq0.
$$
Its hyperbolicity relates to the following Sturm-Liouville operators,
\bea\label{ope-est}
\mathcal{F}(\mathfrak{e},\beta)=-\frac{d^2}{d\theta^2}I_{2}-2J_{2}\frac{d}{d\theta}+\frac{\frac{3}{2}I_{2}}{1+\mathfrak{e}\cos\theta}+\frac{\beta\left(\begin{array}{cccc}-1 & 0\\ 0 & 1\end{array}\right)}{1+\mathfrak{e}\cos \theta}, \ \ \beta>0,
\eea
with domain
$$\bar{D}_{2}(\omega, 2\pi)=\{y\in W^{2,2}([0,2\pi],\mathbb{C}^{2})\,|\,
y(2\pi)=\omega y(0),\dot{y}(2\pi)=\omega\dot{y}(0)\}.$$
\begin{rem}\label{rem-Lag}
This operator is very important, it is related to the Lagrange solution which generated from the Lagrange equilateral triangle central
configuration (see Figure (a) and \cite{HLS}). More precisely, the parameter $\beta$ in
Lagrange solution is taken by $\beta=\frac{\sqrt{9-\beta_{L}}}{2}$, where
$$
\beta_{L}=\frac{27(m_{1}m_{2}+m_{1}m_{3}+m_{2}m_{3})}{(m_{1}+m_{2}+m_{3})^2}\in[0,9]
$$
is the mass parameter depends on the masses $m_{1}, m_{2}, m_{3}$ of the three body in Lagrange solution. It corresponds the following linear Hamiltonian system of Lagrange
solution.
\bea\label{Lag-system}
\gamma'_{\text{Lag}}(\theta)=J_{4}\mathcal{B}_{\text{Lag}}(\theta)\gamma_{\text{Lag}}(\theta),\ \ \gamma_{\text{Lag}}(0)=I_{4}.
\eea
where
\bea
\mathcal{B}_{\text{Lag}}(\theta)=\left(
                                          \begin{array}{cccc}
                                            1 & 0 & 0 & 1 \\
                                            0 & 1 & -1 & 0 \\
                                            0 & -1 & 1-\frac{3-\sqrt{9-\beta_{L}}}{2(1+\mathfrak{e}\cos\theta)} & 0 \\
                                            1 & 0 & 0 & 1-\frac{3+\sqrt{9-\beta_{L}}}{2(1+\mathfrak{e}\cos\theta)} \\
                                          \end{array}
                                        \right).\nonumber
\eea
For the Lagrange solution, we have following useful lemma,
\begin{lem}(\cite{HLS})\label{Hy-Po}
The monodromy matrix $\gamma_{\text{Lag}}(2\pi)$ is hyperbolic for some $(\beta_{L}, \mathfrak{e})$ if and only if the operator
$\mathcal{F}(\mathfrak{e},\frac{\sqrt{9-\beta_{L}}}{2})$ is positive in $\bar{D}_{2}(\omega, 2\pi)$ for any $\omega\in\mathbb{U}$.
\end{lem}
As seen from the correspondence between system (\ref{beta-system}) and system (\ref{Lag-system}), we obtain
\begin{cor}\label{Hy-Po2}
For $\beta\in[0,3/2]$, the monodromy matrix $\gamma_{\beta,\mathfrak{e}}(2\pi)$ is hyperbolic for some $(\beta, \mathfrak{e})$ if and only if the operator
$\mathcal{F}(\mathfrak{e},\beta)$ is positive in $\bar{D}_{2}(\omega, 2\pi)$ for any $\omega\in\mathbb{U}$.
\end{cor}
\end{rem}
In general, it's very difficult to estimate the hyperbolic region of the system (\ref{beta-system}) for any $\mathfrak{e}\in[0,1)$.
The first non-trivial analytic result was obtained by \cite{HLS}, they
prove it's hyperbolicity for $(\beta,\mathfrak{e})\in\{0\}\times[0,1)$, this implies the hyperbolicity of the regular $3$-gon ERE for any eccentricity. In \cite{O}, the first author prove it's hyperbolicity for $(\beta,\mathfrak{e})\in[0,0.5)\times[0,1)$. Further, in \cite{HO}, Hu and the first author further extended the hyperbolic region for $(\beta,\mathfrak{e})\in[0,0.7237)\times[0,1)$
, this implies the hyerpbolicity of the regular $4$-gon ERE for any eccentricity, but this result is not enough to obtain the hyperbolicity of the
regular $5$-gon ERE, hence it is still an unsolved problem. In order to get the hyperbolicity of the regular $5$-gon ERE, we will later see that we need to establish hyperbolicity for $(\beta,\mathfrak{e})\in[0, 1.1459)\times[0,1)$.

While the numerical stability diagram visually suggests extensive hyperbolic regions, numerical methods are fundamentally limited¡ªthey cannot theoretically verify
hyperbolicity over infinite points. Also, the errors of the numerical methods is hard to control when the eccentricity is closed to one, because the linear equation (\ref{beta-system}) has singularity at $\mathfrak{e}=1$. Based on the reason, we further develop the following important lemmas and obtain our Theorem \ref{hyper}, which tells us that if we know
the hyperbolicity of system (\ref{beta-system}) for some fixed points $(\beta_{0}, \mathfrak{e}_{0})$, then we can obtain a large hyperbolic region. By  carefully selecting appropriate parameters $(\beta_{0}, \mathfrak{e}_{0})$, we only need to check the hyperbolicity of the system at finite points $(\beta_{0}, \mathfrak{e}_{0})$.
\begin{lem}\label{lemposi1} If $\mathcal{F}(\mathfrak{e},\beta)$ is positive in $\bar{D}_{2}(\omega,2\pi)$ for some $\beta_{0}>0$ and $\mathfrak{e}_{0}\geq 0$ then
$\mathcal{F}(\mathfrak{e},\beta)$ is positive in $\bar{D}_{2}(\omega,2\pi)$ for any parameter $(\beta, \mathfrak{e})$ which satisfy the following inequality
\bea
0\leq\beta<\frac{1+\mathfrak{e}}{1+3\mathfrak{e}-2\mathfrak{e}_{0}}\beta_{0},\ \ \mathfrak{e}_{0}\leq\mathfrak{e}.\label{esit1}
\eea
\end{lem}
\begin{proof}
Direct computation shows that
\bea
&&\mathcal{F}(\mathfrak{e},\beta)-\frac{\beta}{\beta_{0}}\frac{1+\mathfrak{e}_{0}}{1+\mathfrak{e}}\mathcal{F}(\mathfrak{e}_{0},\beta_{0})\nonumber\\
&=&(1-\frac{\beta}{\beta_{0}}\frac{1+\mathfrak{e}_{0}}{1+\mathfrak{e}})(-\frac{d^2}{d\theta^2}I_{2}-2J_{2}\frac{d}{d\theta})+\frac{(\frac{3}{2}-\beta)I_{2}}{1+\mathfrak{e}\cos\theta}
-\frac{\beta}{\beta_{0}}\frac{1+\mathfrak{e}_{0}}{1+\mathfrak{e}}\frac{(\frac{3}{2}-\beta_{0})I_{2}}{1+\mathfrak{e}_{0}\cos\theta}\nonumber\\
&&+\frac{2\beta\left(\begin{array}{cccc}0 & 0\\ 0 & 1\end{array}\right)}{1+\mathfrak{e}\cos \theta}-
\frac{1+\mathfrak{e}_{0}}{1+\mathfrak{e}}\frac{2\beta\left(\begin{array}{cccc}0 & 0\\ 0 & 1\end{array}\right)}{1+\mathfrak{e}_{0}\cos \theta}.\nonumber
\eea
Since $(\beta, \mathfrak{e})$ satisfy (\ref{esit1}), we have $\frac{\beta}{\beta_{0}}\frac{1+\mathfrak{e}_{0}}{1+\mathfrak{e}}<1$ and $\frac{1+\mathfrak{e}_{0}}{1+\mathfrak{e}}\leq\frac{1+\mathfrak{e}_{0}\cos\theta}{1+\mathfrak{e}\cos\theta}$ hence
\bea\label{inequ1}
&&\mathcal{F}(\mathfrak{e},\beta)-\frac{\beta}{\beta_{0}}\frac{1+\mathfrak{e}_{0}}{1+\mathfrak{e}}\mathcal{F}(\mathfrak{e}_{0},\beta_{0})\nonumber\\
&\geq&(1-\frac{\beta}{\beta_{0}}\frac{1+\mathfrak{e}_{0}}{1+\mathfrak{e}})(-\frac{d^2}{d\theta^2}I_{2}-2J_{2}\frac{d}{d\theta}+
\frac{\frac{3}{2}-\beta-\frac{\beta}{\beta_{0}}(\frac{3}{2}-\beta_{0})}{1-\frac{\beta}{\beta_{0}}\frac{1+\mathfrak{e}_{0}}{1+\mathfrak{e}}}\frac{I_{2}}{1+\mathfrak{e}\cos\theta})\nonumber\\
&=&(1-\frac{\beta}{\beta_{0}}\frac{1+\mathfrak{e}_{0}}{1+\mathfrak{e}})(-\frac{d^2}{d\theta^2}I_{2}-2J_{2}\frac{d}{d\theta}+
\frac{\frac{3}{2}(1-\frac{\beta}{\beta_{0}})}{1-\frac{\beta}{\beta_{0}}\frac{1+\mathfrak{e}_{0}}{1+\mathfrak{e}}}\frac{I_{2}}{1+\mathfrak{e}\cos\theta})
\eea
Direct computation shows that
\bea\label{decom}
R(\theta)(-\frac{d^2}{d\theta^2}I_{2}-2J_{2}\frac{d}{d\theta}+
\frac{\frac{3}{2}(1-\frac{\beta}{\beta_{0}})}{1-\frac{\beta}{\beta_{0}}\frac{1+\mathfrak{e}_{0}}{1+\mathfrak{e}}}\frac{I_{2}}{1+\mathfrak{e}\cos\theta})R^{T}(\theta)
=\mathcal{A}(\delta_{1},\mathfrak{e})\oplus\mathcal{A}(\delta_{1},\mathfrak{e}),
\eea
where $\mathcal{A}(\delta_{1},\mathfrak{e})$ is given by (\ref{1-dim-oper}) with $$\delta_{1}=\frac{\frac{3}{2}(1-\frac{\beta}{\beta_{0}})}{1-\frac{\beta}{\beta_{0}}\frac{1+\mathfrak{e}_{0}}{1+\mathfrak{e}}}.$$
If $(\beta, \mathfrak{e})$ satisfy
$$
0\leq\beta<\frac{1+\mathfrak{e}}{1+3\mathfrak{e}-2\mathfrak{e}_{0}}\beta_{0},\ \ \mathfrak{e}_{0}\leq\mathfrak{e},
$$
then $\delta_{1}>1$, from Lemma \ref{lem-hyp}, we have $\mathcal{A}(\delta_{1},e)>0$ in $\bar{D}_{1}(\omega, 2\pi)$, combine with (\ref{inequ1}) and (\ref{decom}), we obtain
$$
\mathcal{F}(\mathfrak{e},\beta)>\frac{\beta}{\beta_{0}}\frac{1+\mathfrak{e}_{0}}{1+\mathfrak{e}}\mathcal{F}(\mathfrak{e}_{0},\beta_{0}),\ \ \text{in}\ \ \bar{D}_{2}(\omega, 2\pi),
$$
hence the positivity of $\mathcal{F}(\mathfrak{e}_{0},\beta_{0})$ implies the positivity of $\mathcal{F}(\mathfrak{e},\beta)$. This completes the proof of the lemma.
\end{proof}
\begin{lem}\label{lem-posi2} If $\mathcal{F}(\mathfrak{e},\beta)$ is positive in $\bar{D}_{2}(\omega,2\pi)$ for some $\beta_{0}>0$ and $\mathfrak{e}_{0}\geq 0$ then
$\mathcal{F}(\mathfrak{e},\beta)$ is positive in $\bar{D}_{2}(\omega,2\pi)$ for any parameter $(\beta, \mathfrak{e})$ which satisfy the following inequality
\bea
0\leq\beta<\frac{1-\mathfrak{e}}{1-3\mathfrak{e}+2\mathfrak{e}_{0}}\beta_{0},\ \ \mathfrak{e}_{0}\geq\mathfrak{e}.\label{esti2}
\eea
\end{lem}
\begin{proof}
Direct computation shows that
\bea
&&\mathcal{F}(\mathfrak{e},\beta)-\frac{\beta}{\beta_{0}}\frac{1-\mathfrak{e}_{0}}{1-\mathfrak{e}}\mathcal{F}(\mathfrak{e}_{0},\beta_{0})\nonumber\\
&=&(1-\frac{\beta}{\beta_{0}}\frac{1-\mathfrak{e}_{0}}{1-\mathfrak{e}})(-\frac{d^2}{d\theta^2}I_{2}-2J_{2}\frac{d}{d\theta})+\frac{(\frac{3}{2}-\beta)I_{2}}{1+\mathfrak{e}\cos\theta}
-\frac{\beta}{\beta_{0}}\frac{1-\mathfrak{e}_{0}}{1-\mathfrak{e}}\frac{(\frac{3}{2}-\beta_{0})I_{2}}{1+\mathfrak{e}_{0}\cos\theta}\nonumber\\
&&+\frac{2\beta\left(\begin{array}{cccc}0 & 0\\ 0 & 1\end{array}\right)}{1+\mathfrak{e}\cos \theta}-
\frac{1-\mathfrak{e}_{0}}{1-\mathfrak{e}}\frac{2\beta\left(\begin{array}{cccc}0 & 0\\ 0 & 1\end{array}\right)}{1+\mathfrak{e}_{0}\cos \theta}.\nonumber
\eea
Since $(\beta, \mathfrak{e})$ satisfies (\ref{esti2}), we have $\frac{\beta}{\beta_{0}}\frac{1-\mathfrak{e}_{0}}{1-\mathfrak{e}}<1$ and
$\frac{1-\mathfrak{e}_{0}}{1-\mathfrak{e}}\leq\frac{1+\mathfrak{e}_{0}\cos\theta}{1+\mathfrak{e}\cos\theta}$, hence
\bea\label{inequ2}
&&\mathcal{F}(\mathfrak{e},\beta)-\frac{\beta}{\beta_{0}}\frac{1-\mathfrak{e}_{0}}{1-\mathfrak{e}}\mathcal{F}(\mathfrak{e}_{0},\beta_{0})\nonumber\\
&\geq&(1-\frac{\beta}{\beta_{0}}\frac{1-\mathfrak{e}_{0}}{1-\mathfrak{e}})(-\frac{d^2}{d\theta^2}I_{2}-2J_{2}\frac{d}{d\theta}+
\frac{\frac{3}{2}-\beta-\frac{\beta}{\beta_{0}}(\frac{3}{2}-\beta_{0})}{1-\frac{\beta}{\beta_{0}}\frac{1-\mathfrak{e}_{0}}{1-\mathfrak{e}}}\frac{I_{2}}{1+\mathfrak{e}\cos\theta})\nonumber\\
&=&(1-\frac{\beta}{\beta_{0}}\frac{1-\mathfrak{e}_{0}}{1-\mathfrak{e}})(-\frac{d^2}{d\theta^2}I_{2}-2J_{2}\frac{d}{d\theta}+
\frac{\frac{3}{2}(1-\frac{\beta}{\beta_{0}})}{1-\frac{\beta}{\beta_{0}}\frac{1-\mathfrak{e}_{0}}{1-\mathfrak{e}}}\frac{I_{2}}{1+\mathfrak{e}\cos\theta})
\eea
Direct computation shows that
\bea\label{decom2}
R(\theta)(-\frac{d^2}{d\theta^2}I_{2}-2J_{2}\frac{d}{d\theta}+
\frac{\frac{3}{2}(1-\frac{\beta}{\beta_{0}})}{1-\frac{\beta}{\beta_{0}}\frac{1-\mathfrak{e}_{0}}{1-\mathfrak{e}}}\frac{I_{2}}{1+\mathfrak{e}\cos\theta})R^{T}(\theta)
=\mathcal{A}(\delta_{2},\mathfrak{e})\oplus\mathcal{A}(\delta_{2},\mathfrak{e}),
\eea
where $\mathcal{A}(\delta_{2},\mathfrak{e})$ is given by (\ref{1-dim-oper}) with $$\delta_{2}=\frac{\frac{3}{2}(1-\frac{\beta}{\beta_{0}})}{1-\frac{\beta}{\beta_{0}}\frac{1-\mathfrak{e}_{0}}{1-\mathfrak{e}}}.$$
If $(\beta, \mathfrak{e})$ satisfy
$$
0\leq\beta<\frac{1-\mathfrak{e}}{1-3\mathfrak{e}+2\mathfrak{e}_{0}}\beta_{0},\ \ \mathfrak{e}_{0}\geq\mathfrak{e},
$$
then $\delta_{2}>1$, from Lemma \ref{lem-hyp}, we have $\mathcal{A}(\delta_{2},e)>0$ in $\bar{D}_{1}(\omega, 2\pi)$, combine with (\ref{inequ2}) and (\ref{decom2}), we obtain
$$
\mathcal{F}(\mathfrak{e},\beta)>\frac{\beta}{\beta_{0}}\frac{1-\mathfrak{e}_{0}}{1-\mathfrak{e}}\mathcal{F}(\mathfrak{e}_{0},\beta_{0}),\ \ \text{in}\ \ \bar{D}_{2}(\omega, 2\pi),
$$
hence the positivity of $\mathcal{F}(\mathfrak{e}_{0},\beta_{0})$ implies the positivity of $\mathcal{F}(\mathfrak{e},\beta)$. This completes the proof of the lemma.
\end{proof}
Based on Corollary \ref{Hy-Po2}, Lemma \ref{lemposi1} and Lemma \ref{lem-posi2}, we can easy to obtain Theorem \ref{hyper}.
\begin{proof}
Take finite set $\mathcal{K}=\{(1.36,0.0), (1.36,0.1),(1.36,0.2),\ldots,(1.36,0.9)\}$, numerically, one can check $\gamma_{\beta,\mathfrak{e}}(2\pi)$ is hyperbolic for $(\beta_{0},\mathfrak{e}_{0})\in\mathcal{K}$.
From Corollary \ref{Hy-Po2}, we know $\mathcal{F}(\mathfrak{e}_{0},\beta_{0})$ is positive, this combine with Lemma \ref{lemposi1} and Lemma\ref{lem-posi2}, we get the positive
definite region $\mathfrak{U}$ for $\mathcal{F}(\mathfrak{e},\beta)$, use Corollory  \ref{Hy-Po2} again, we get $\gamma_{\beta,\mathfrak{e}}(2\pi)$ is hyperbolic in region $\mathfrak{U}$.
This completes the proof of Theorem \ref{hyper}.
\end{proof}

Finally, based on Theorem \ref{hyper}, we will give a unified proof of the hyperbolicity of the regular $3$, $4$ and $5$-gon ERE for any eccentricity $\mathfrak{e}\in[0,1)$.
\subsection{Hyperbolicity of the regular $3, 4$ and $5$-gon ERE}
Indeed, the hyperbolicity of the regular $3$-gon follows straightforwardly from the perspective of our Theorem \ref{redu-n-gon}, since the essential part of the regular $3$-gon only has one part $\eta_{1,\mathfrak{e}}(\theta)$. The hyperbolicity of $\eta_{1,\mathfrak{e}}(2\pi)$ is already implied by our Theorem \ref{unstable}, hence the regular $3$-gon is hyperbolic for any $\mathfrak{e}\in[0,1)$. For the regular $4$ and $5$-gon ERE, the essential part can be decomposed into two parts $\eta_{1,\mathfrak{e}}(\theta), \eta_{2,\mathfrak{e}}(\theta)$, we only need to study the hyperbolicity of $\eta_{2,\mathfrak{e}}(2\pi)$. First, for the regular $4$-gon ERE, $\eta_{2,\mathfrak{e}}(\theta)$ satisfies a four dimensional Hamiltonian system,
$$
\begin{aligned}
& \eta_{2,  \mathfrak{e}}^{\prime} (\theta) =J \mathcal{B}_{2} (\theta) \eta_{2,  \mathfrak{e}} (\theta),   \quad \eta_{2,  \mathfrak{e}} (0) =I_{4} . \\
& \mathcal{B}_{2} (\theta) =\left (\begin{array}{cc}
I_{2} & -J_{2} \\
J_{2} & I_{2}-\frac{I_{2}+\mathcal{U}_{2}}{1+\mathfrak{e} \cos \theta}
\end{array}\right),
\end{aligned}
$$
where
$$
\mathcal{U}_{2}=\frac{1}{\lambda}\left (\begin{array}{cccc}
a_{2} & 0  \\
0 & b_{2}
\end{array}\right),
$$
$$
a_{2}=P_{2}-3 Q_{2}=,  \quad b_{2}=P_{2}+3 Q_{2},
$$
$$
P_{2}=\sum_{j=1}^{3} \frac{1-\cos \theta_{j 2} \cos \theta_{j}}{2 d_{4 j}^{3}},  \quad Q_{2}=\sum_{j=1}^{3} \frac{\cos \theta_{j}-\cos \theta_{j 2}}{2 d_{4 j}^{3}},  \quad \lambda=\frac{1}{4} \sum_{j=1}^{3} \csc \frac{\pi j}{4} .
$$
Direct computations show that
$$
\frac{a_{2}}{\lambda}=\frac{2\sqrt{2}-4}{4+\sqrt{2}},\ \ \frac{b_{2}}{\lambda}=\frac{8-\sqrt{2}}{4+\sqrt{2}}.
$$
We rewrite $\mathcal{U}_{2}$ in the following form,
$$
\begin{aligned}
\mathcal{U}_{2}&=\frac{a_{2}+b_{2}}{2\lambda}I_{2}+\frac{b_{2}-a_{2}}{2\lambda}\left (\begin{array}{cccc}
-1 & 0  \\
0 & 1
\end{array}\right)\\
&=\frac{1}{2}I_{2}+\frac{12-3\sqrt{2}}{8+2\sqrt{2}}\left (\begin{array}{cccc}
-1 & 0  \\
0 & 1
\end{array}\right).
\end{aligned}
$$
This is just the system (\ref{hyper}) with $\beta=\frac{12-3\sqrt{2}}{4+\sqrt{2}}\approx0.7164$, since $\{0.7164\}\times[0,1)\subset\mathfrak{U}$, Theorem \ref{hyper}
implies the hyperbolicity of the regular $4$-gon ERE for any eccentricity.

For the regular $5$-gon ERE, it satisfies a eight dimensional Hamiltonian system, it's more complex than the $3, 4$-gon system. More precisely, it's satisfies
$$
\eta'_{2,\mathfrak{e}}(\theta)=J\mathcal{B}_2(\theta)\eta_{2,e}(\theta),\ \ \eta_{2,\mathfrak{e}}(0)=I_{8}.
$$
\bea
\mathcal{B}_{2}(\theta)=\left( \begin{array}{cccc} I_{4} & -\mathbb{J}_{2}\\
\mathbb{J}_{2} & I_{4}-\frac{I_{4}+\mathcal{U}_{2}
}{1+\mathfrak{e}\cos\theta}  \end{array}\right),\nonumber
\eea
where
\bea
\mathcal{U}_{2}==\frac{1}{\lambda}\left(\begin{array}{cccc}a_{2} & 0 & 0 & S_{2}\\
0 & b_{2} & -S_{2} & 0\\ 0 & -S_{2} & a_{2} & 0\\
S_{2} & 0 & 0 & b_{2}\end{array}\right),\nonumber
\eea
\bea
a_{2}=P_{2}-3Q_{2},\ \ b_{2}=P_{2}+3Q_{2},\nonumber
\eea
\bea\label{parameter}
P_{2}=\sum_{j=1}^{4}\frac{1-\cos \theta_{j2}\cos \theta_{j}}{2d_{5j}^3},\ \
S_{2}=\sum_{j=1}^{4}\frac{\sin \theta_{j2}\sin \theta_{j}}{2d_{5j}^3},\ \
Q_{2}=\sum_{j=1}^{4}\frac{\cos \theta_{j}-\cos \theta_{j2}}{2d_{5j}^3},\ \
\lambda=\frac{1}{4}\sum_{j=1}^{4}\csc \frac{\pi j}{5}.
\eea
In order to prove the hyperbolicity of the regular $5$-gon ERE, we need further estimation. Let consider the following operator
$$
\mathcal{A}_{2}(\mathfrak{e})=-\frac{d^{2}}{d\theta^{2}}-2\mathbb{J}_{2}\frac{d}{d\theta}+\frac{I_{4}+\mathcal{U}_{2}}{1+\mathfrak{e}\cos \theta},\ \ \mathfrak{e}\in[0,1),$$
which is a self-adjoint operator with domain
$$\bar{D}_{4}(\omega, 2\pi)=\{y\in W^{2,2}([0,2\pi],\mathbb{C}^{4})\,|\,
y(2\pi)=\omega y(0),\dot{y}(2\pi)=\omega\dot{y}(0)\},$$
where $\omega\in\mathbb{U}$ and $W^{2,2}([0,2\pi])$ is the usual Sobolev space.

Now, we define
\bea
E_{2}=\left( \begin{array}{cccc} a_{2} & 0\\
0 & b_{2}\end{array}\right),\ \ F_{2}=\left( \begin{array}{cccc} 0 & S_{2}\\
-S_{2} & 0\end{array}\right),\ \ G_{2}=\left( \begin{array}{cccc} b_{2} & 0\\
0 & a_{2}\end{array}\right),\ \ \tilde{F}_{2}=\left(\begin{array}{cccc} S_{2} & 0\\
0 & S_{2}\end{array}\right)\nonumber
\eea
then
\bea
\mathcal{A}_{2}(\mathfrak{e})=
-\frac{d^2}{d\theta^2}I_{4}-2\mathbb{J}_{2}\frac{d}{d\theta}+\frac{I_{4}+\frac{1}{2\lambda}\left(\begin{array}{cccc}E_{2}+G_{2} & 2F_{2}\\
-2F_{2} & E_{2}+G_{2}\end{array}\right)}{1+\mathfrak{e}\cos \theta}+\frac{\frac{1}{2\lambda}\left(\begin{array}{cccc}E_{2}-G_{2} & 0\\
0 & E_{2}-G_{2}\end{array}\right)}{1+\mathfrak{e}\cos \theta}\nonumber
\eea
Direct computation shows that $S_{2}\approx 0.262865>0$, then in $\bar{D}_{4}(\omega,2\pi)$ we have
\bea
\left(\begin{array}{cccc}E_{2}+G_{2}-2\tilde{F}_{2} & 0\\
0 & E_{2}+G_{2}-2\tilde{F}_{2}\end{array}\right)\leq\left(\begin{array}{cccc}E_{2}+G_{2} & 2F_{2}\\
-2F_{2} & E_{2}+G_{2}\end{array}\right).\nonumber
\eea
hence we have
\bea\label{inequl2}
\check{\mathcal{A}}_{2}(\mathfrak{e})\leq\mathcal{A}_{2}(\mathfrak{e}),\ \ \mathrm{in}\ \ \bar{D}_{4}(\omega,2\pi),
\eea
where
\bea
&&\check{\mathcal{A}}_{2}(\mathfrak{e})=
-\frac{d^2}{d\theta^2}I_{4}-2\mathbb{J}_{2}\frac{d}{d\theta}+\frac{I_{4}+\frac{1}{2\lambda}\left(\begin{array}{cccc}E_{2}+G_{2}-2\tilde{F}_{2} & 0\\
0 & E_{2}+G_{2}-2\tilde{F}_{2}\end{array}\right)}{1+\mathfrak{e}\cos \theta}+\frac{\frac{1}{2\lambda}\left(\begin{array}{cccc}E_{2}-G_{2} & 0\\
0 & E_{2}-G_{2}\end{array}\right)}{1+\mathfrak{e}\cos \theta} \nonumber
\eea
It is obviously that this two operator can be decomposed,
\bea
\check{\mathcal{A}}_{2}(\mathfrak{e})=\check{\mathcal{A}}_{2,0}(\mathfrak{e})\oplus\check{\mathcal{A}}_{2,0}(\mathfrak{e}),\label{dec}
\eea
where
\bea
\check{\mathcal{A}}_{2,0}(\mathfrak{e})&=&-\frac{d^2}{d\theta^2}I_{2}-2J_{2}\frac{d}{d\theta}+\frac{I_{2}+\frac{1}{2\lambda}(E_{2}+G_{2}-2\tilde{F}_{2})}{1+\mathfrak{e}\cos \theta}
+\frac{\frac{1}{2\lambda}(E_{2}-G_{2})}{1+\mathfrak{e}\cos \theta}\nonumber\\
&=&-\frac{d^2}{d\theta^2}I_{2}-2J_{2}\frac{d}{d\theta}+\frac{I_{2}+\frac{1}{2\lambda}(a_{2}+b_{2}-2S_{2})I_{2}}{1+\mathfrak{e}\cos \theta}
+\frac{\frac{1}{2\lambda}(b_{2}-a_{2})\left(\begin{array}{cccc}-1 & 0\\ 0 & 1\end{array}\right)}{1+\mathfrak{e}\cos \theta}, \nonumber
\eea
From (\ref{parameter}), direct computation shows that
$$
\frac{1}{2\lambda}(a_{2}+b_{2}-2S_{2})=\frac{1}{2}, \ \ \frac{1}{2\lambda}(b_{2}-a_{2})\approx 1.145898.
$$
One can see that $\check{\mathcal{A}}_{2,0}(\mathfrak{e})$ is just the operator (\ref{ope-est}) with $\beta=\frac{1}{2\lambda}(b_{2}-a_{2})\approx 1.145898$.
Since $\{1.145898\}\times[0,1)\subset \mathfrak{U}$. Theorem \ref{hyper} and Corollary \ref{Hy-Po2} implies $\check{\mathcal{A}}_{2,0}(\mathfrak{e})$ is positive on
$$\bar{D}_{2}(\omega, 2\pi)=\{y\in W^{2,2}([0,2\pi],\mathbb{C}^{2})\,|\,
y(2\pi)=\omega y(0),\dot{y}(2\pi)=\omega\dot{y}(0)\}$$
for any eccentricity and $\omega\in\mathbb{U}$. From (\ref{dec}) and (\ref{inequl2}), we know $\mathcal{A}_{2}(\mathfrak{e})$ is positive on $\bar{D}_{4}(\omega, 2\pi)$ for any eccentricity and $\omega\in\mathbb{U}$. Direct computation shows that if the monodromy matrix $\eta_{2,\mathfrak{e}}(2\pi)$
has eigenvalue $\omega\in \mathbb{U}$, then there must exist $0\neq x\in \bar{D}_{4}(\omega, 2\pi)$ such that $\mathcal{A}_{2}(\mathfrak{e})x=0$, hence the positive definiteness $\mathcal{A}_{2}(\mathfrak{e})$ in $\bar{D}_{4}(\omega, 2\pi)$ for any $\omega\in\mathbb{U}$ implies
the monodromy matrix $\eta_{2,\mathfrak{e}}(2\pi)$ has no eigenvalues on $\mathbb{U}$, and thus $\eta_{2,\mathfrak{e}}(2\pi)$ is hyperbolic. This completes the proof the hyperbolicity
of the regular $5$-gon ERE for any eccentricity.

\hfill\newline
\textbf{Acknowledgments}
\noindent{\bf Acknowledgement.}
The authors thank Prof. Xijun Hu for useful discussions.
Y.Ou is partially supported by NSFC ($\sharp$12371192), the Young Taishan
Scholars Program of Shandong Province ($\sharp$ tsqn202312055), and the Qilu Young
Scholar Program of Shandong University.

\medskip


\end{document}